\documentclass[10pt]{article}
\usepackage[table]{xcolor}
\usepackage[english]{babel}			
\usepackage[utf8]{inputenc}		
\usepackage[T1]{fontenc}
\usepackage{amsmath,amsfonts,amssymb,amsthm,cancel,siunitx,calculator,calc,mathtools,empheq,latexsym}
\usepackage{epsfig,tikz,float}
\usepackage{booktabs,multicol,multirow,tabularx,array}
\usepackage{eurosym}
\usepackage{graphicx}
\usepackage{subcaption}
\usepackage{dsfont}
\usepackage{longtable}
\usepackage{listings}
\usepackage{comment}
\usepackage{float}
\usepackage{hyperref}
\usepackage{csquotes}
\usepackage[linesnumbered,lined,boxed,commentsnumbered]{algorithm2e} 
\RestyleAlgo{ruled}

\newcommand{\R}{\mathbb{R}}

\newcommand{\E}{\mathbb{E}}

\def\red#1{{\color{red}#1}}

\def\gr#1{{\color{green}#1}}

\def\Yc{{\cal Y}}
\def\Zc{{\cal Z}}

\def\Fc{{\cal F}}
\def\Uc{{\cal U}}

\def\Wc{{\cal W}}
\def\Jc{{\cal J}}

\usepackage[
backend=bibtex,
style=alphabetic,sorting=nty, maxbibnames=99
]{biblatex}

\definecolor{codegreen}{rgb}{0,0.6,0}
\definecolor{codegray}{rgb}{0.5,0.5,0.5}
\definecolor{codepurple}{rgb}{0.58,0,0.82}
\definecolor{backcolour}{rgb}{0.95,0.95,0.92}
\lstdefinestyle{mystyle}{
    backgroundcolor=\color{backcolour},   
    commentstyle=\color{codegreen},
    keywordstyle=\color{magenta},
    numberstyle=\tiny\color{codegray},
    stringstyle=\color{codepurple},
    basicstyle=\ttfamily\footnotesize,
    breakatwhitespace=false,         
    breaklines=true,                 
    captionpos=b,                    
    keepspaces=true,                 
    numbers=left,                    
    numbersep=5pt,                  
    showspaces=false,                
    showstringspaces=false,
    showtabs=false,                  
    tabsize=2
}
\usepackage{tikz, pgfplots}
\newcommand{\overbar}[1]{\mkern 1.5mu\overline{\mkern-1.5mu#1\mkern-1.5mu}\mkern 1.5mu}
\allowdisplaybreaks
\usetikzlibrary{positioning}
\usepackage[left= 3cm]{geometry}
\newtheorem{remark}{Remark}
\newtheorem{theorem}{Theorem}[section]
\newtheorem{assumption}{Assumption}[section]

\theoremstyle{definition}
\newtheorem{definition}{Definition}[section]
\title{\normalsize\bf%
\vspace{2cm}
\uppercase{ Deep learning algorithms for FBSDEs with jumps: Applications to option pricing and a MFG model for smart grids \thanks{T\lowercase{his work has received financial support from  \uppercase{FIME
(F}inance for \uppercase{E}nergy \uppercase{M}arkets) research initiative of the \uppercase{I}nstitut \uppercase{E}uroplace de \uppercase{F}inance, which is gratefully acknowledged.}}}
}
    
\author{Clemence Alasseur \footnote{EDF R\&D \& FiME, 91120 Palaiseau, France,  email: \texttt{clemence.alasseur@edf.fr}} \and Zakaria Bensaid \footnote{
Laboratoire Manceau de Mathématiques \& FR CNRS N\textsuperscript{o} 2962, Institut du Risque et de l’Assurance, Université Le Mans, France, email: \texttt{zakaria.bensaid@univ-lemans.fr}} \and Roxana Dumitrescu \thanks{Department of Mathematics, King's College London, United Kingdom, email: \texttt{roxana.dumitrescu@kcl.ac.uk}} \and Xavier Warin \footnote{EDF R\&D \& FiME, 91120 Palaiseau, France, email: \texttt{xavier.warin@edf.fr}}}

\pgfplotsset{compat=1.18}
\begin{document}

\date{}

\maketitle

\vspace{-0.5cm}

% -------------------------------------------------------------------
% Abstract
\bigskip
\noindent
{\small{\bf ABSTRACT.} In this paper, we introduce various  machine learning solvers for (coupled) forward-backward systems of stochastic differential equations (FBSDEs) driven by a Brownian motion and a Poisson random measure. We provide a rigorous comparison of the different algorithms and demonstrate their effectiveness in various applications, such as cases derived from pricing with jumps and mean-field games. %with finite and infinite activity.
In particular, we show the efficiency of the deep-learning algorithms to solve a coupled multi-dimensional FBSDE system driven by a time-inhomogeneous jump process with stochastic intensity, which describes the Nash equilibria for a specific mean-field game (MFG) problem for which we also provide the complete theoretical resolution.  More precisely, we develop an extension of the MFG model for smart grids introduced in \cite{MFG_revised} to the  case when the random jump times correspond to the jump times of a doubly Poisson process. We first provide an existence result of an equilibrium and derive its semi-explicit characterization in terms of a multi-dimensional FBSDE system  in the linear-quadratic setting. We then compare the MFG solution to the optimal strategy of a central planner and provide several numerical illustrations using the deep-learning solvers presented in the first part of the paper.  

%developed in \cite{han2018solving, chan2019machine, hure2020deep, germain2020deep, andersson2022convergence} and a mean-field game model with common noise and jumps for smart grids \cite{MFG_revised}. The extension considers the resolution of backward stochastic differential equations with Poisson jumps and doubly stochastic Poisson jumps \cite{Cox_1955, Lewis_1978, Miyazawa_1987}, leading to a more accurate representation of the system. Numerical simulations are performed to compare our methods with existing ones and demonstrate its effectiveness in various applications, such as pricing with jumps with finite and infinite activity, and also solving a MFG characterized by a system of two coupled FBSDEs with two independent Brownian motions and two doubly stochastic Poisson processes with intensities adapted to each Brownian motion. Finally, we provide some numerical results for the smart grids system and study the MFG impact through the price of anarchy.   \\}


\baselineskip=\normalbaselineskip
% -----------------------------------
\section{Introduction}

This paper is devoted to the numerical resolution of a coupled system of forward-backward stochastic differential equations (in short FBSDEs) with jumps of the form:
\begin{align}\label{sto}
\begin{cases}
    dX_{t} =  b(t,X_{t},Y_{t})dt + \sigma(t,X_{t}) dW_{t} + \int_{\mathbb{R}^{d}\setminus \{ 0 \}} \beta (t,X_{t^-}, e) \Tilde{\mathcal{J}}(dt,de), \\
    dY_{t} = -f(t,X_{t},Y_{t}) dt + Z_{t}dW_{t} + \int_{\mathbb{R}^{d}\setminus \{ 0 \}} U_{t}(e) \mathcal{\tilde{J}}(dt,de)\red{,}\\
    X_{0} = \xi, \quad Y_{T} = g(X_{T})\red{,}
    \end{cases}
    t \in [0,T],
\end{align}
where the functions $b,\sigma, \beta,f, g$, as well as the the initial condition $\xi$ satisfy appropriate assumptions which ensure the well-posedness of the system $\eqref{sto}$.

This kind of equations are linked to a class of  (deterministic) partial integro-differential equations, which are non-local and take the following form:
\begin{align}\label{det}
\begin{cases}
\frac{\partial u}{\partial t}(t,x) + \mathcal{L} u(t,x) + f\left(t, x, u(t,x) \right) = 0, \quad (t, x) \in [0, T) \times \mathbb{R}^{d},\\
u(T, x) = g(x), \quad x \in \mathbb{R}^{d},
\end{cases}
\end{align}
where the second-order nonlocal operator $\mathcal{L}$ is defined as follows:

\begin{align*}
\mathcal{L} u(t,x) = \langle b(t, x, u(t,x)), D_{x} u(t,x) \rangle + \frac{1}{2} \langle  D_{xx}^2 u(t,x) \sigma(t, x), \sigma(t,x) \rangle \\ 
+ \int_{\mathbb{R}^{d}\setminus \{ 0 \}} \left(u(t, x + \beta(t, x, e)) - u(t, x) - \langle D_{x} u(t,x), \beta(t, x, e) \rangle \right) \nu(de).
\end{align*}
Indeed, it is known that, under mild assumptions, $Y_t=u(t,X_t)$, where $u$ corresponds to the viscosity solution of \eqref{det}.
We refer to \cite{pardoux1990adapted} for a rigorous connection between PDEs and FBSDEs in a Markovian setting in the case of decoupled system of FBSDEs and Brownian filtration, further extended to the case with jumps in \cite{BP97} and to the case of coupled FBSDEs with jumps in \cite{zhen1999, wu2003fully}.

\textit{Literature review.} The resolution of partial differential equations (in short PDEs) by standard techniques as finite difference methods becomes unfeasible beyond dimension $3$. An alternative method to solve nonlinear PDEs in dimension above $4$  is based
on the  backward stochastic differential equation (in short BSDE) representation of semilinear PDEs: using the time discretization
scheme proposed in \cite{BT04}, some effective algorithms based on regression have been developed in \cite{gobet2005regression}, \cite{lemor2006rate}
and has led to a lot of research as shown for example in \cite{gobet2016linear}. This regression technique uses some basis functions that can be either some global polynomials as in \cite{longstaff2001valuing} or some local polynomials as proposed in \cite{bouchard2012montecarlo}: therefore this methodology still faces the curse of dimensionality and can only solve some problems in dimension below $7$ or $8$. 

Over the past few years, machine learning methods have shown exceptional promise to solve high-dimensional nonlinear PDEs (see e.g. \cite{Darbon_2016, han2018solving, chan2019machine}). Machine learning methods have emerged since the pioneering papers by
\cite{han2018solving} and \cite{sirignano2018dgm}, and have shown their efficiency for solving high-dimensional nonlinear PDEs by
means of neural networks approximation.
\cite{sirignano2018dgm} proposes
the so-called Deep Galerkin Method which uses the automatic numerical differentiation of the solution to solve the PDE on a finite domain. The authors prove the
convergence of their method,  but without information on the rate of convergence. An alternative methodology to solve PDEs in high-dimension is based on the BSDE representation of the solution of the PDE and deep learning approximations (see e.g. \cite{han2018solving, BEJ19, HL20, Ji20, germain2020deep}).
Two main classes of algorithms have been developed.  The first class is based on the \textit{global} approach, which was initially proposed in \cite{han2018solving} to tackle semi-linear PDEs. It
consists in the training of as many neural networks as time steps by solving in a forward way the backward representation of the PDE solution. The $Z_t$ process is represented by a different
neural network $Z_{i}^{\theta}$ with parameters $\theta$ at each date $t_i$. Instead of solving the BSDE starting from
the terminal condition, the method writes it down as a forward equation and an optimization
problem aiming to reach the terminal condition $g(X_T)$ by minimizing a mean squared error $\mathbb{E}[|Y_T-g(X_T)|^2]$. It allows to solve PDEs in high dimension and a convergence study of Deep BSDE
is conducted in \cite{HL20}. In \cite{BEJ19}, this approach has been extended to fully nonlinear equations. Furthermore, \cite{chan2019machine} shows that using a single network across all dates is more efficient, and additionally introduces a fixed-point algorithm to resolve semi-linear PDEs.

The second class of algorithms is based on the  \textit{local} approach, first proposed in \cite{hure2020deep}, which consists in solving local optimization problems at each time step in a backward manner. Unlike the global method, the local method involves successive optimization problems of moderate dimension. At each time step, local neural networks are trained, thus it results in as many learning problems as time steps with two  neural networks (in the setting of a Brownian filtration). This process is further simplified by utilizing strategies inspired from the standard backward resolution of BSDEs with conditional expectations and regression techniques from \cite{BT04, gobet2005regression, lemor2006rate, bender2007forward}. The resulting solver was named the Deep Backward Dynamic Programming (DBDP) solver. Furthermore, the methodology is then expanded to handle fully nonlinear PDEs in \cite{PWG21} by merging it with ideas proposed in \cite{Bec19}. Additionally, extensive tests performed in \cite{hure2020deep} indicate that the local method yields superior results compared to the global one, such as \cite{PWG21} for fully nonlinear dynamics. A more robust machine learning solver, called deep backward multi-step scheme (MDBDP), was introduced in \cite{germain2020deep} that builds on the Linear Regression Multi step-forward Dynamic Programming (MDP) scheme for discrete BSDEs introduced in \cite{gobet2016linear}. According to the authors, the multi-step scheme yields the best performance when compared to other algorithms in the \textit{local} approach.

Machine learning techniques to solve coupled FBSDEs within a Brownian filtration are explored in \cite{HL20} and \cite{Ji20}. The resulting algorithms  are all rooted in the \textit{global} approach first introduced in \cite{han2018solving}. 

The resolution of partial integro-differential equations (in short PIDEs) has been much less regarded in the literature, even in the decoupled case, the main approaches to solve them being based on the finite-difference methods (see e.g. \cite{PIDE05volt}) and  the probabilistic representation of the solution in terms of a FBSDE system, a discrete time approximation of the associated decoupled forward-backward SDE with jumps being proposed in \cite{BE08}. As it can be noticed above, the literature on machine learning solvers for standard PDEs is quite rich by now. In contrast, the case of integro-differential PDEs has received very little attention.  Several algorithms have been recently proposed in: \cite{bayraktar2022neural, frey2022convergence,Gnoatto22, liu2023study}). In these papers,  the deep-learning solvers are based on the approximation of the solution of the PIDE and, for the gradient, either another neural network is employed \cite{bayraktar2022neural, frey2022convergence, liu2023study}, or the Automatic differentiation in TensorFlow is applied \cite{Gnoatto22}.}

\textit{Contributions.}  The aim of our paper is to develop deep-learning solvers for the (coupled) FBSDE system \eqref{sto} and a specific multi-dimensional coupled FBSDE system driven by a time-inhomogeneous Poisson process with stochastic intensity which is shown to solve an extended version of the MFG model in \cite{MFG_revised}. Our main contributions are the following:
\begin{itemize}
\item In the first part of the paper, we introduce five different algorithms to solve the system \eqref{sto}. %with possibly infinite activity jumps.
Furthermore, we propose two different variants of the DBDP and MDBDP solvers to handle the jumps part. We emphasize that most of the literature on deep learning solvers for FBSDEs with jumps does not treat the fully coupled case and that the algorithms developed in this paper are new also in the context of decoupled FBSDEs with jumps.
\item We provide a rigorous numerical comparison between all methods in terms of computation time, stability and convergence for different pricing models, which require solving a decoupled FBSDE system. To assess the performance of the algorithms for coupled FBSDEs,  we purposefully introduce an equivalent coupled FBSDE system in the context of pricing which uses the explicit form of the analytical solution already known. This allows to benchmark the deep learning solvers in different settings.
\item In the second part of the paper, we develop an extension of the MFG model introduced in \cite{MFG_revised} to the case when the random jump times correspond to the jumps times of a doubly Poisson process. We first provide an existence result of an equilibria by using the stochastic maximum principle and derive its semi-explicit characterization in terms of a multi-dimensional coupled FBSDE system driven by a Cox process in the linear-quadratic setting. We then compare the MFG solution to the optimal strategy of a central planner. 

\item We build a numerical algorithm based on the deep-learning solvers presented in the first part of the paper to solve the multi-dimensional coupled FBSDE systems driven by a Cox process, which characterize the Nash equilibrium for the MFG problem and the mean-field optimal control of the central planner. In particular, we numerically demonstrate the robustness of our deep learning-based numerical methods in handling time-inhomogeneous jump processes with stochastic intensity.
\end{itemize}
 The paper is organized as follows: in Section 2.1, we give some Preliminaries on existence and uniqueness results for (coupled) FBSDEs with jumps and on neural networks. In Section 2.2, we introduce the five different deep learning solvers for (coupled) FBSDEs with jumps. In Section 2.3, we perform several numerical tests for examples derived from option pricing and provide a detailed analysis and comparison between the different algorithms. In Section 3.1, we present the mean-field game model. In Section 3.2, we characterize the mean-field equilibria, and in Section 2.3 we study the related problem of a central planner and characterize the mean-field optimal control (MFC). Section 3.4. is devoted to the implementation of the Deep learning solvers for the multi-dimensional fully coupled FBSDE system characterizing the MFG equilibria (and the MFC optimal control), as well as to the comparison between the different algorithms. Finally, in Section 3.5 we provide an interpretation of the numerical results from an economic perspective.

\section{General Deep Learning algorithms for coupled FBSDEs with jumps} 
This section is devoted to the presentation of different deep learning algorithms for coupled FBSDEs with jumps and of their performance on several numerical examples. We shall start with some preliminaries. 
\subsection{Preliminaries}\label{sec:ML}

\,\,\,\,\,\,\, In this subsection, we first introduce some notation, as well as some existence and uniqueness results related to coupled FBSDEs with jumps. We then focus on neural networks which are used to solve numerically the FBSDE system.

%paper, we will focus on Forward Backward Stochastic Differential Equations with jumps. We will see later in the mathematical approach that the extended MFG problem resulting out of our smart grids model  is solved using the stochastic maximum principle, which leads to a system of FBSDEs with jumps. \\

\underline{\textit{Coupled FBSDEs with jumps}}. Fix a time horizon $T < \infty$ and let $(\Omega,  \mathbb{F}, \mathbb{P})$ be a complete probability space. Let $W_t$ be a $d$-dimensional Brownian motion and $\mathcal{J}(dt,de)$ an independent Poisson random measure with compensator $\nu(de)dt$ such that $\nu(de)$ is a $\sigma$-finite measure on $\mathbb{R}^{d}\setminus \{ 0 \}$, equipped with its Borel field $\mathcal{B}(\mathbb{R}^{d}\setminus \{ 0 \})$. Let $\mathcal{\tilde{J}}$ be the compensated jump measure, i.e. $\tilde{\mathcal{J}}(dt, de) := \mathcal{J}(dt, de) - \nu(de)dt$. Let $\mathbb{F} = (\mathcal{F}_t)_{t \in [0, T]}$ be the (completed) filtration associated with $W$ and $\mathcal{J}$. Assume that $\nu$ satisfies the condition
\begin{equation*}
    \int_{\mathbb{R}^{d}\setminus \{ 0 \}} (1 \wedge |e|^2) \nu(de) < \infty.
\end{equation*}
We now introduce the following spaces, using the usual inner product and the Euclidean norm in $\mathbb{R}^d, \mathbb{R}^k$ and $\mathbb{R}^{k \times d}$, respectively.
\begin{itemize}
    \item $L^2(\mathcal{G}, \R^d)$ is the set of $\R^d$-valued random variables $\xi$ which are $\mathcal{G}$-measurable such that $\mathbb{E}\left [|\xi|^2 \right] < + \infty$, where $\mathcal{G}$ is a sub-$\sigma$-algebra of $\mathcal{F}_T$.
    \item $\mathcal{S}^2$ is the set of $\mathbb{F}$-adapted càdlàg $\R^k$-valued processes $Y$ such that $\mathbb{E}\left [\sup_{0 \leq t \leq T} |Y_t|^2 \right] < + \infty$.
    \item $\mathcal{H}^2$ is the set of $\mathbb{F}$-predictable $\R^{k \times d}$-valued processes $Z$ such that $\| Z \|^{2} : = \mathbb{E}[\int_{0}^{T} |Z_{t}|^{2} dt] < + \infty$.
    %$\mathbb{F}$-predictable processes (I don't think it is the case)
    \item $\mathcal{H}_{\nu}^2$ is the set of $\mathcal{P} \otimes \mathcal{B}(\R^d \setminus \{0\})$-measurable maps $U$ taking values in $\R^k$ such that $\| U \|^{2}_{\nu} : = \mathbb{E}[\int_{0}^{T} \int_{\mathbb{R}^{d}\setminus \{ 0 \}} |U_{t}(e)|^{2} \nu(de) dt] < + \infty$, where $\mathcal{P}$ denotes the $\sigma$-field of $\mathbb{F}$-predictable subsets of $\Omega \times [0,T]$.
\end{itemize}

\noindent We now introduce the following coupled FBSDE system with jumps:
\begin{align}\label{eqn:FBSDEDL}
\begin{cases}
    dX_{t} =  b(t,X_{t},Y_{t})dt + \sigma(t,X_{t}) dW_{t} + \int_{\mathbb{R}^{d}\setminus \{ 0 \}} \beta (t,X_{t^-}, e) \Tilde{\mathcal{J}}(dt,de), \\
    dY_{t} = -f(t,X_{t},Y_{t}) dt + Z_{t}dW_{t} + \int_{\mathbb{R}^{d}\setminus \{ 0 \}} U_{t}(e) \mathcal{\tilde{J}}(dt,de)\red{,}\\
    X_{0} = \xi, \quad Y_{T} = g(X_{T}) \red{,}
    \end{cases}
    t \in [0,T],
\end{align}
where the functions $b :[0,T] \times \mathbb{R}^{d} \times \mathbb{R}^{k}  \rightarrow \mathbb{R}^{d} $, $\sigma : [0,T] \times \mathbb{R}^{d} \rightarrow \mathbb{R}^{d}$, $\beta: [0,T] \times \mathbb{R}^{d} \times \mathbb{R}^{d}\setminus \{ 0 \} \rightarrow \mathbb{R}^{d} $,  $ f :[0,T] \times \mathbb{R}^{d} \times \mathbb{R}^{k}  \rightarrow \mathbb{R}^{k}$ and $g : \mathbb{R}^{d} \rightarrow \mathbb{R}^{k}$ are measurable maps which have to satisfy the following assumptions ensuring the well-posedness of the FBSDE system.

%We then focus on the conditions and hypotheses for our FBSDE, which are based on the functions $b$, $\sigma$, $\beta$, $f$, and $g$. These conditions include uniform Lipschitz continuity with respect to $(x, y)$ for the drift and the driver and uniform Lipschitz continuity with respect to $x$ for the volatility and terminal condition, we also present additional assumptions on the functions and their domains. Let us outline these conditions and hypotheses.
\begin{assumption}\label{assume:Regularity}
    \begin{enumerate}
    \item[(i)] $b$ and $f$ are uniformly Lipschitz with respect to $(x, y)$, and there exists $\rho: \mathbb{R}^{d}\setminus \{ 0 \} \rightarrow \mathbb{R}^+$ with $\int_{\mathbb{R}^{d}\setminus \{ 0 \}} \rho^2(e) \nu(de) < +\infty$ such that for any $t \in [0, T]$, $x, \bar{x} \in \mathbb{R}^d$,  and $e \in \mathbb{R}^{d}\setminus \{ 0 \}$, we have:
    \begin{align*}
    &\lvert \beta (t, x, e) - \beta (t, \bar{x}, e)\rvert \leq \rho(e) \lvert x - \bar{x}\rvert .
    \end{align*}
    
    \item[(ii)] $\sigma$ and $g$ are uniformly Lipschitz with respect to $x \in \mathbb{R}^d$.
    
    \item[(iii)] Furthermore,
    \begin{align*}
    &  \int_0^T \lvert b(s, 0, 0)\rvert^2 ds  + \int_0^T  \lvert f(s, 0, 0)\rvert^2 ds + \int_0^T  \int_{\mathbb{R}^{d}\setminus \{ 0 \}} \lvert \beta(s, 0, e)\rvert^2 \nu(de) ds < \infty.
    \end{align*}
\end{enumerate}
\end{assumption}

Given an $k \times d$ full-rank matrix $G$ and $G^T$ being the transposed matrix of $G$, we define:

\begin{align*}
\pi = \begin{pmatrix} x \\ y \end{pmatrix} \mbox{ in } \R^d \times \R^k , \quad A(t, \pi) = \begin{pmatrix} -G^T f\\ Gb  \end{pmatrix}(t, \pi) \mbox{ in } \R^d \times \R^k.
\end{align*}
For any $\pi = (x, y)$, and $\bar{\pi} = (\bar{x}, \bar{y})$, let us denote  $ \tilde{x} = x - \bar{x}$, and $\tilde{y} = y - \bar{y}$.
We also assume the following monotonicity conditions hold .
\begin{assumption} \label{assume:Monotonicity}
 There exists $\gamma_1, \gamma_2, \mu_1$ non negative constants with $\gamma_1 + \gamma_2 > 0$, $\gamma_2 + \mu_1 > 0$, such that
   \begin{enumerate}
\item[(i)] $\langle A(t, \pi) - A(t, \bar{\pi}), \pi - \bar{\pi} \rangle \leq -\gamma_1 |G \tilde{x}|^2 - \gamma_2 |G^{T} \tilde{y}|^2$.
\item[(ii)] $\langle g(x) - g(\bar{x}), G(x - \bar{x}) \rangle \geq \mu_1 |G \tilde{x}|^2$, %for any $\pi = (x, y) \in \mathbb{R}^{d} \times \mathbb{R}^{k}, \bar{\pi} = (\bar{x}, \bar{y}) \in \mathbb{R}^{d} \times \mathbb{R}^{k}, \tilde{x} = x - \bar{x}, \tilde{y} = y - \bar{y}$,
%where $\gamma_1, \gamma_2, \mu_1$ are nonnegative constants with $\gamma_1 + \gamma_2 > 0$, $\gamma_2 + \mu_1 > 0$. Moreover, we have $\gamma_1 > 0, \mu_1 > 0$ (respectively, $\gamma_2 > 0$) when $k > d$ (respectively, $k < d$).
\end{enumerate}
Moreover, we have $\gamma_1 > 0, \mu_1 > 0$ (respectively, $\gamma_2 > 0$) when $k > d$ (respectively, $k < d$).
 \end{assumption}
 \begin{assumption}\label{assume:time}
    Assume that $b$, $f$, $\sigma$ and $\beta$ are uniformly $\frac{1}{2}$-Hölder continuous in time.
 \end{assumption}
\begin{assumption} \label{assume:small}
    Assume that $k = 1$ and for any $t \in [0,T]$, $(x,y) \in \mathbb{R}^{d + 1}$, we have:
    $$\lvert b(t,x,y) \lvert +  \lvert f(t,x,y) \lvert + \lvert \sigma(t,x) \lvert + \lvert g(x) \lvert \leq K (1 + \lvert x \lvert + \lvert y\lvert),$$
    and there exists $\rho: \mathbb{R}^{d}\setminus \{ 0 \} \rightarrow \mathbb{R}^+$ with $\int_{\mathbb{R}^{d}\setminus \{ 0 \}} \rho^2(e) \nu(de) < +\infty$ such that for any $t \in [0, T]$, $x \in \mathbb{R}^d$,  and $e \in \mathbb{R}^{d}\setminus \{ 0 \}$, we have
    $\lvert \beta(t,x,e) \lvert \leq \rho(e)(1 + \lvert x \lvert)$.
\end{assumption} 
We now give two well-posedness results for the FBSDE \eqref{eqn:FBSDEDL}, as well as a decoupling field representation of the $Y$-component of the system, which follow from \cite{zhen1999}, \cite{LI20141582}, and \cite{BE08}.

 \begin{theorem}[\textit{Well-posedness for arbitrary time horizon}]
     Under the Assumptions \ref{assume:Regularity} and \ref{assume:Monotonicity}, for any $\xi \in L^2(\mathcal{F}_0, \mathbb{R}^d)$, FBSDE (\ref{eqn:FBSDEDL}) has a unique  solution $(X, Y, Z, U) \in \mathcal{S}^2 \times \mathcal{S}^2 \times \mathcal{H}^2 \times \mathcal{H}^2_{\nu}$. 
 \end{theorem}
We give here an alternative existence and uniqueness result for a fully-coupled FBSDE in small time.

  \begin{theorem}[\textit{Well-posedness in small time}] \label{thm:smallmat}
     Under the Assumptions \ref{assume:Regularity} and \ref{assume:small}, there exists a constant $\delta_0 >0$ only depending on $\rho$, $K$, and the Lipschitz constants of $b$, $\sigma$, and $f$ such that for every $0 \leq \delta \leq \delta_0$, and $\xi \in L^2(\mathcal{F}_{t} ,\mathbb{R}^d)$, FBSDE (\ref{eqn:FBSDEDL}) has a unique  solution $(X_s, Y_s, Z_s, U_s)_{s \in [t, t + \delta]}$ on the time interval $[t, t+\delta]$. 
 \end{theorem}

  Let us introduce the decoupling field:
 $$ u(t, x) = Y_{s}^{t, x}|_{s = t}, \quad (t, x) \in [0, T] \times \mathbb{R}^d,$$
 where $Y^{t, x}$ is the solution of the FBSDE \eqref{eqn:FBSDEDL} with the initial condition $X_t = x$. Using the Markov property of the forward component $X$  of the system \eqref{eqn:FBSDEDL} and the continuity of the function $u$ with respect to $x$, it is shown in \cite{LI20141582} that, under the above assumptions, for any $(t, \xi) \in [0,T] \times L^2(\mathcal{F}_t, \mathbb{R}^d)$, we have 
\begin{align}\label{decoupl}
Y^{t,\xi}_s = u(s, X_s), \,\, t \leq s \leq T,  \quad \mathbb{P}-{\rm a.s.}
\end{align}
where $X$ is the solution of the SDE with initial state $\xi$ at time $t$ and $Y^{t,\xi}$ the associated backward component  of the FBSDE system \eqref{eqn:FBSDEDL}. Furthermore, under the given assumptions on the coefficients, $u$ is uniformly Lipschitz, and has linear growth with respect to $x \in \mathbb{R}^d$. Finally, by the assumption \eqref{assume:time}, we recover the $\frac{1}{2}$-Hölder continuity of the decoupling field $u$ with respect to time which is essential for the discrete approximation discussed later.
 %\begin{theorem}[Decoupling field] \label{decouplingfield}
% Under the above assumptions, for any $(t, \xi) \in [0,T] \times L^2(\mathbb{R}^d)$, we have $$ Y_t = u(t, X_t), \quad \mathbb{P}-a.s.,$$
%where $Y$ and $X$ are solutions of the FBSDE \eqref{eqn:FBSDEDL} with initial state $\xi$. Furthermore, $u$ is uniformly Lipschitz, and has linear growth with respect to $x \in \mathbb{R}^d$.
% \end{theorem}
We also have the following representation for the component $U$ of the solution:
for all $(t, e) \in[0,T]\times \mathbb{R}^d \setminus\{0\}$, $$U_t(e) = u(t, X_{t^{-}} + \beta(t, X_{t^{-}}, e)) - u(t, X_{t^-}), \quad \mathbb{P} - \text{a.s.}$$
%\begin{lemma} \label{representationU}
%    Under the above assumptions, if $(X, Y, Z, U) \in \mathcal{S}^2 \times \mathcal{S}^2 \times \mathcal{H}^2 \times \mathcal{H}^2_{\nu}$ is a solution of FBSDE \eqref{eqn:FBSDEDL}, we have, for all $(t, e) \in[0,T]\times \mathbb{R}^d \setminus\{0\}$, $$U_t(e) = u(t, X_{t^{-}} + \beta(t, X_{t^{-}}, e)) - u(t, X_{t^-}), \quad \mathbb{P} - a.s.$$
%\end{lemma}
%\noindent The proof is provided in a fully coupled setting in \cite{zhen1999}.\\

%Consequently, we present various universal techniques for solving FBSDEs with jumps, utilizing ideas from the works of \cite{han2018solving, chan2019machine, hure2020deep, germain2020deep} in the context of diffusion. We demonstrate their efficacy through multiple applications, particularly for pricing in exponential-Levy models and mean-field games with shared noise and doubly stochastic Poisson jumps.\\
The second part of the preliminaries concentrates on a brief introduction to neural networks.\\

\underline{\noindent \textit{Neural networks.}} We consider a feedforward neural network, denoted by $\Phi^{\theta}$, which approximates the processes of interest. Let $d_0$ be the input dimension, and $d_1$ be the output dimension. We fix an integer $L \geq 2$ to represent the total number of layers, including the input and output layers. We define $m$ to be the number of neurons on each hidden layer, and for simplicity, we set $m_0 = d_0$ and $m_{L - 1} = d_1$. 

The feedforward neural network is defined as the composition of affine transformations and nonlinear activation functions. Specifically, we have:
\begin{equation*}
\Phi^{\theta} = A_{L} \circ \sigma_{a} \circ A_{L-1} \circ \cdots \circ \sigma_{a} \circ A_{1},
\end{equation*}
where $\sigma_a$ is a component-wise activation function,  $A_1$ is a mapping  from $\mathbb{R}^{d_0}$ to $ \mathbb{R}^m$, $A_L$ is a mapping from $\mathbb{R}^m$ to $\mathbb{R}^{d_1}$ and for $l =2$ to $L-1$,  $A_l$ is a mapping from  $\mathbb{R}^m$ to $\mathbb{R}^m$. % is an affine function mapping  $\mathbb{R}^{d_{l-1}}$ to $\mathbb{R}^{d_{l}}$,
%and each $\sigma_a$ is a component-wise activation function.

We represent each affine function $A_l$ as $A_l(x) = W_l x + \beta_l$, where $W_l$ is a matrix of weights and $\beta_l$ is a vector of biases.

The neural network has parameters $\theta$, which include all the weights and biases of the affine functions. The total number of parameters is $N^L_{d_0, m,d_1} = (d_0 +1)m+(L-2)m(1+m)+(m+1)d_1$, 
%$N_m = d_0 (1+m)+(L-2)m(1+m)+m(1+d_1)$
where $m$ is the number of neurons on each hidden layer.

We denote by $\mathcal{N}\mathcal{N}_{\infty}$ the set of such functions $\Phi^{\theta}$. To restrict the number of neurons per layer, we introduce $\mathcal{N}\mathcal{N}_{p}$ the set of neural networks with at most $p \in \mathcal{N}$ neurons per hidden layer and $L-1$ hidden layers. We recall here the two following approximation theorems.
\begin{theorem}[Universal Approximation Theorem]
Assume that the function $\sigma_{a}$ is non constant and bounded. Let $\mu$ denote a probability measure on $\mathbb{R}^{d}$, then for any $L \geq 2$, $\mathcal{N}\mathcal{N}_{\infty}$ is dense in $L^{2}(\mathbb{R}^{d},\mu)$.
\end{theorem}

\begin{theorem}[Universal Approximation Theorem]
Assume that the function $\sigma_{a}$ is non constant, bounded and a continuous function, then when $L = 2$, $\mathcal{N}\mathcal{N}_{\infty}$ is dense in $C(\mathbb{R}^{d})$ for the topology of the uniform convergence on compact sets.
\end{theorem}

\subsection{Deep Learning algorithms}\label{alg}
\quad\,  We introduce here five \textit{deep learning algorithms} to solve the coupled system of forward-backward SDEs with jumps \eqref{sto} in the case of jumps with finite activity (i.e. $\lambda = \int_{\mathbb{R}^d \setminus \{0\}} \nu(de)< \infty$).

%present here 5 \textbf{deep learning algorithms} to solve forward-backward SDEs with jumps, utilizing the ideas from the approaches developed, in the Brownian setting, in \cite{han2018solving,  chan2019machine, hure2020deep, germain2020deep}.

%\subsubsection{The \textit{finite} activity case.}

Let us first define the L\'evy process $J$ associated with the Poisson random measure $\mathcal{J}$, which is given, for $0 \leq t \leq T$, by
\begin{align}\label{jump}
J_t:=\int_0^t \int_{\mathbb{R}^d \setminus \{0\}} e \mathcal{J}(ds,de).
\end{align}
and 
$\Delta J_t:=J_t-J_{t^-}$, for all $t>0$.
We also introduce the following Poisson process, denoted by $N_t$:
\begin{align*}
N_t:=\int_0^t \int_{\mathbb{R}^d \setminus \{0\}}  \mathcal{J}(ds,de).
\end{align*}
The Poisson Process $(N_t)$ has the intensity $\lambda t$.

By defining  the map $\bar{b}(t,x,y):=b(t,x,y)-\int_{\mathbb{R}^d \setminus \{ 0 \}} \beta(t,x,e) \nu(de)$, we observe that the FBSDE \eqref{eqn:FBSDEDL} system can be written as:
\begin{align}\label{eqn:FBSDEDL1}
\begin{cases}
    dX_{t} =  \bar{b}(t,X_{t},Y_{t})dt + \sigma(t,X_{t}) dW_{t} + \int_{\mathbb{R}^{d}\setminus \{ 0 \}} \beta (t,X_{t^-}, e) \mathcal{J}(dt,de), \\
    dY_{t} = -f(t,X_{t},Y_{t}) dt + Z_{t}dW_{t} + \int_{\mathbb{R}^{d}\setminus \{ 0 \}} U_{t}(e) \mathcal{\tilde{J}}(dt,de) \red{,}\\
    X_{0} = \xi, \quad Y_{T} = g(X_{T}) \red{.}
    \end{cases}
    t \in [0,T],
\end{align}

\begin{remark}
    The numerical approximation of the compensator of the jump part of the forward component in the drift $\bar{b}$ can be done through numerous methods. For example, explicit integration with respect to the intensity measure, Monte Carlo estimation, or the methods used in \cite{PIDE05volt}.
\end{remark}

By using \eqref{decoupl}, the FBSDE system \eqref{eqn:FBSDEDL1} reads as follows:
\begin{align}\label{eqn:FBSDEDL2}
\begin{cases}
    X_{t} =  \xi+ \int_0^t \bar{b}(s,X_{s},u(s,X_s))ds + \int_0^t \sigma(s,X_{s}) dW_{s} + \int_0^t\int_{\mathbb{R}^{d}\setminus \{ 0 \}} \beta (s,X_{s^-}, e) \mathcal{J}(ds,de), \\
    u(t, X_t) = g(X_T)-\int_t^T f(s,X_{s},u(s,X_s)) ds + \int_t^TZ_{s}dW_{s} + \int_t^T \int_{\mathbb{R}^{d}\setminus \{ 0 \}} U_{s}(e) \mathcal{\tilde{J}}(ds,de).\\
    \end{cases}
    t \in [0,T].
\end{align}
\paragraph{\textit{Discrete-time approximation}.}
Let us consider a uniform time grid $\pi:=\{t_{0}, t_{1}, ..., t_{M} \}$ where $t_{i} := i \frac{T}{M}$ for $i \in \{0, 1, ..., M\}$ and $\Delta t_i := t_{i+1} - t_{i}$ represents the constant time step size. We also define the Brownian increment $\Delta W_i$ as $\Delta W_i := W_{t_{i+1}} - W_{t_{i}}$ and the Poisson increment $\Delta N_{i} := N_{t_{i + 1}} - N_{t_{i}}$, which follows a Poisson distribution with mean $\lambda \Delta t_{i}$. Finally, for a fixed %$i \in \{0,1,\ldots,M\}$
\ $i \in \{0,1,\ldots,M-1\}$, denote by $(\Delta J^i_l)_{l \in [1, \Delta N_{i}]}$  the $l$th jump of the process $(J_t)$ given by \eqref{jump} which occurs on the time interval $]t_{i}, t_{i+1}]$.\\

\noindent To give the intuition about the approximation of the backward SDE of the system \eqref{eqn:FBSDEDL2} and, in particular, about the treatment of the jump part, we  first introduce the following continuous-time process $(\bar{X}^{\pi}_t)_{t \in [0,T]}$:
 $$ \bar{X}_t^{\pi} := X_{t_i} + \int_{t_i}^t\int_{\mathbb{R}^{d}\setminus \{ 0 \}} \beta(t_i, \bar{X}_{t_i}^{\pi},e) \Jc(ds,de), \quad \forall t \in [t_i, t_{i + 1}[, \forall i \in [\lvert 0, M-1\lvert]\red{,}$$
and the process $(\bar{U}_t^\pi)$ which is defined as follows
 $$\bar{U}^\pi_t(e):=u(t_i,\bar{X}^{\pi}_{t^-}+\beta(t_i, \bar{X}^{\pi}_{t_{i}},e))-u(t_i,\bar{X}^{\pi}_{t^-}), \quad \forall t \in [t_i, t_{i + 1}[, \forall i \in [\lvert 0, M-1\lvert].$$
 We can observe that, for a number of time steps $M$ sufficiently large, we have the following approximation:
 \begin{align*}
 u(t_{i}, X_{t_i})  \approx  u(t_{i+1}, X_{t_{i + 1}}) + f(t_i, X_{t_i}, u(t_i, X_{t_i}))\Delta t_{i}  - \bar{Z}^{\pi}_i \Delta W_i- \int_{t_i}^{t_{i+1}}\int_{\mathbb{R}^d \setminus \{0\}} \bar{U}^\pi_s(e)\Tilde{\mathcal{J}}(ds,de),
 \end{align*}
with $\bar{Z}^{\pi}_i:=\frac{1}{\Delta t_i}\mathbb{E}[\int_{t_i}^{t_{i+1}} Z_s ds \lvert \Fc_{t_{i}}] \approx \mathbb{E}\left [u(t_{i+1},X_{t_{i+1}}) \frac{\Delta W_{i}}{\Delta t_i}|\mathcal{F}_{t_i} \right] $. Note that the integral of $(\bar{U}^\pi_t)$ with respect to the Poisson measure $\mathcal{J}$ admits the  representation:
\begin{align*}
        \int_{t_i}^{t_{i+1}}\int_{\mathbb{R}^d \setminus \{0\}} \bar{U}^\pi_s(e)\mathcal{J}(ds,de) &=\int_{t_i}^{t_{i+1}} \int_{\mathbb{R}^d \setminus \{0\}} u(t_i,\bar{X}^{\pi}_{s^-}+\beta(t_i, \bar{X}^{\pi}_{t_{i}},e))-u(t_i,\bar{X}^{\pi}_{s^-}) \mathcal{J}(ds,de) \\
      \text{(by definition of $\bar{X}^{\pi}$)}  &= \displaystyle \sum \limits ^{\Delta N_{i}}_{k=1} u(t_i, X_{t_i} + \displaystyle \sum \limits ^{k}_{l=1} \beta(t_i, X_{t_i}, \Delta J_l^{i})) - u(t_i, X_{t_i} + \displaystyle \sum \limits ^{k-1}_{l=1} \beta(t_i, X_{t_i}, \Delta J_l^{i})),\\
        &= u(t_i, X_{t_i} + \displaystyle \sum \limits ^{\Delta N_{i}}_{l=1} \beta(t_i, X_{t_i}, \Delta J_l^{i})) - u(t_i, X_{t_i}).
    \end{align*}
\noindent For simplicity, we well denote from now on $X_{t_i}$ by $X_{i}$. By using the Euler scheme to approximate the solution $X_t$ of the SDE, we are led to the following discrete time approximation of the solution of the FBSDE system \eqref{eqn:FBSDEDL1}:
\begin{align}\label{eqn:discrete1}
\begin{cases}
      X^{\pi}_{i+1} =  X^{\pi}_{i} + \bar{b}(t_i,X^{\pi}_i, u(t_{i}, X_{i}^\pi)) \Delta t_{i}  + \sigma (t_i,X^{\pi}_i)   \Delta W_i + \displaystyle \sum \limits ^{\Delta N_{i}}_{l=1} \beta(t_i, X_i^{\pi}, \Delta J_l^{i}), \\
      u(t_{i}, X_{i}^\pi)  \approx  u(t_{i+1}, X^\pi_{i+1}) + f(t_i,X^{\pi}_{i} , u(t_{i}, X_{i}^\pi))\Delta t_{i}  - Z^{\pi}_i \Delta W_i
      - \left(u(t_i, X^{\pi}_i + \displaystyle \sum \limits ^{\Delta N_{i}}_{l=1} \beta(t_i, X^{\pi}_i, \Delta J_l^{i})) - u(t_i, X^{\pi}_i)\right) \\ + \mathbb{E} \left [ u(t_i, X^{\pi}_i + \displaystyle \sum \limits ^{\Delta N_{i}}_{l=1} \beta(t_i, X^{\pi}_i, \Delta J_l^{i})) - u(t_i, X^{\pi}_i) \Big |\mathcal{F}_{t_i} \right], \\
            Z^{\pi}_{i} = \mathbb{E}\left [u(t_{i+1},X^\pi_{i+1}) \frac{\Delta W_{i}}{\Delta t_i}|\mathcal{F}_{t_i} \right],\\
      X^{\pi}_0 = \xi, \quad u(t_M, X^\pi_{M}) = g(X^{\pi}_M),\\
      i = 0, \cdots, M-1.
\end{cases}
\end{align}
%We are therefore led to consider the following discrete-time approximation of the FBSDE system (\ref{eqn:FBSDEDL}), based on the classical Euler scheme, the approximation of the jumps part for the backward component above, and the best $L^2$ approximation between $t_i$ and $t_{i + 1}$ for the process $Z$ as defined in \cite{BE08}:

%\red{ You cannot subtract integrals in all algorithms. It has to be discrete. Or analytical. Same in all algorithms. Question: do we have  to use the compensated version for X. In fact, it is used in test cases and it is only a pachange of the drift in these cases.
%}

\begin{remark}

\begin{itemize}
%\item Note that this discrete scheme will be adapted for the deep learning solvers introduced below, which do not use the semi-explicit expression of $Z_i^{\pi}$.  Instead, due to the Markovian setting, $Z_t$ can be expressed as a deterministic function of $t$ and $X_t$, and it will be approximated by a neural network on a fixed time grid.
\item The approximation proposed here for the jumps part for the backward component is different from the one proposed in \cite{Gnoatto22} and is particularly well-suited for the deep learning framework. In this approach, the neural networks deal with the sum of the jumps rather than handling each jump individually which can be problematic for large intensities given a small number of time steps. The numerical tests that we have conducted show the robustness of this approximation. \footnote{The full convergence of the algorithms proposed in this paper  will be  provided in an upcoming paper.} 
\item Regarding the approximation of the component $Z_{i}^\pi$, we use the representation provided in \eqref{eqn:discrete1}, which allows to express the $Z^\pi_t$ component  as a function $v(t, X^\pi_t)$ and to approximate it by a neural network.
 \item To handle the case when the Poisson increment $\Delta N$ equals $0$, we adopt the convention $\displaystyle \sum \limits_{l = 1}^{0} \phi_l := 0$, where $(\phi)_l$ represents a sequence of random variables.
%\item Numerical comparisons have been conducted to test the robustness of this choice. We concluded that when approximating the jumps part in the backward component by the best $L^2$ approximation constant between $t_i$ and $t_{i + 1}$, the time grid depends on the parameters of the Lévy measure as opposed to our setting, where we can take a small number of time steps for large intensities. (\textcolor{magenta}{explain here why, in particular the fact that taking the sum of the jumps is optimal to avoid giving a variable number of parameters representing the individual jumps and allows to take a small number of time steps for large intensities })
   
\end{itemize}

\end{remark}
\begin{comment}
\textcolor{magenta}{Estimate in continuous time
\begin{align}
\sum_{t_i}\mathbb{E}\left[\left( \sum_{t_i \leq T_l<t_{i+1}}(u(T_l,X_{T_l^-}+\beta(T_l,X_{T_l^-},\Delta J_{l}))-u(T_l,X_{T_l^{-}})-(u(t_i,X_{t_{i}}+\sum_{j \leq l}\beta(t_i,X_{t_i},\Delta J_j)))-u(t_{i}, X_{t_{i}}+\sum_{j < l}\beta(t_i,X_{t_i},\Delta J_j)))\right)^2\right]
\end{align}
The indexes are not written very well, but this is the idea. Try to use, besides the Lipschitz property of $\beta$ the estimates  on SDES of type $\mathbb{E}[\sup_{t\leq s \leq t+h} |X_s-X_t|^2] \leq f(h)..$ -- check the exact estimates using le poly de Bouchard sur le controle}
%This convention also applies to the sequence $$$(\Delta J_l)_{l \in [1,dN]}$ when $dN = 0$.
\end{comment}
\paragraph{\textit{From the discrete-time FBSDE to neural networks}.} 
We denote by $\Uc^{\theta}$ the network function approximating the decoupling field $u$, $\Zc^{\theta}$ the network function approximating the deterministic map $t,x \mapsto v(t,x)$ which is used to represent the process  $Z$ (see the remark above), and $\Wc^{\theta}$ the network function  approximating the function $t,x,y \to u(t, x + y) - u(t, x)$. We present two families of algorithms: one is based on the representation $\eqref{eqn:discrete1}$ (below denoted by first class of algorithms) and the second one relies on the regression methods (denoted by second class of algorithms).

\subsubsection{First class  of deep-learning algorithms.}
In this part, we introduce the deep-learning algorithms based on the representation $\eqref{eqn:discrete1}$ (with possibly two variants depending on the algorithm).\\

\noindent \textbf{1. Global solver}. This algorithm  extends to the case of jumps and fully coupled setting the \textit{Deep BSDE} solver developed in \cite{han2018solving}, where each neural network takes  $t$ (i.e. time) as input (see \cite{chan2019machine}).  In our setting, we use  three networks:  $\Yc^{\theta}$ to approximates the initial condition of the backward component, $\Zc^\theta$ to approximate the control $Z$ and $\Wc^{\theta}$  to approximate the jump part in equation \eqref{eqn:discrete1} leading to 
\begin{align*}
Y_{i+1}^{\pi} \approx  Y_i^{\pi} -f(t_i,  X_i^{\pi}, Y_i^{\pi}) \Delta t_i + \Zc^\theta(t_i, X_i^{\pi}) \Delta W_i  +
   \Wc^\theta(t_i, X_i^{\pi} , \sum_{l = 1}^{\Delta N_{i}} \tilde{\beta}_i(\Delta J^{i}_l))    - 
  \mathbb{E} \left[
   \Wc^\theta(t_i, X_i^{\pi},  \sum_{l = 1}^{\Delta N_{i}}  \tilde{\beta}_i(\Delta J^{i}_l)) \big \lvert \mathcal{F}_{t_i} \right],
\end{align*}
where $\tilde{\beta}_i(\Delta J^{i}_l) = \beta(t_i, X_i^{\pi}, \Delta J^{i}_l)$.
Notice that the network $\Wc^{\theta}$ has to depend on $t$, $X^{\pi}$, and $\sum_s \tilde{\beta}_s(\Delta J_s)$. 

Observe that we have the following result to compute the conditional expectation by means of Monte Carlo on each trajectory of the batch:
\begin{align*}
    &\mathbb{E} \left[\Wc^\theta(t_i, X_i^{\pi},  \sum_{l = 1}^{\Delta N_{i}}  \tilde{\beta}_i(\Delta J^{i}_l)) \big \lvert \mathcal{F}_{t_i} \right] = \Theta(t_i,X_i^{\pi}),\\
    \text{where} \qquad &\Theta(t,x) =  \mathbb{E} \left[ \Wc^\theta(t, x, \sum_{l = 1}^{\Delta N_{i}}  \beta(t,x,\Delta J^{i}_l)) \right], \quad \forall (t,x) \in [0,T] \times \mathbb{R}^d.
\end{align*}    
Thus, we choose small batch sizes during the gradient descent in order to estimate the compensator with a large number of Monte Carlo simulations for each sample of the batch.

Let $ \theta = (\theta_0,\theta_1, \theta_2)$, where $\theta_0 \in \R^{N^L_{d,m,k}}$ are the parameters of the network function $\Yc^{\theta_0}$, $\theta_1 \in \R^{N^L_{d+1,m,kd}}$, are the parameters of the network function $\Zc^{\theta_1}$ %from $\R \times \R^d$
, $\theta_2 \in \R^{N^L_{2d+1,m, k}}$ are the parameters of the network function  $\Wc^{\theta_2}$. 
This method consists in training the neural networks by solving in a forward way the backward representation of the solution, i.e. instead of solving the BSDE starting from the terminal condition, one estimates $Y_{0}$ with $\Yc^{\theta_0}(\xi)$ and solves the forward optimization problem with the aim of minimizing $\mathbb{E} \left [ \lvert Y_{T} - g(X_{T})\rvert^{2} \right]$. The \textit{Global solver}  is detailed in Algorithm \ref{alg:globalDirac}.

\begin{algorithm}[H]
\For{$ m = 0, \dots, K$}{
%Get a trainable random initial condition $y_{0}$ using a kernel initializer.\\
$\forall j \in [|1,B|]$  sample $\xi_j$ from the law of $\xi$, and set $X_{0}^{j}(\theta) =\xi_j$, $Y_{0}^{j}(\theta)  = \Yc^{\theta_0}(\xi_j) $ \;  
\For{$i= 0, \dots, M-1$}
{\For{$j = 0, \dots, B$ }%{(carried out in parallel to estimate the compensator)}{ 
{
Sample $\Delta W_i^j$ from a Gaussian vector, sample $\Delta N^{j}_{i}$ from a Poisson distribution $\mathcal{P}\left (\lambda \Delta t_{i} \right)$ and sample each element of the jumps sequence $(\Delta J_l^{i,j})_{l = 1, \cdots, \Delta N^{j}_{i}}$ from the distribution  $\frac{\nu(de)}{\lambda} \mathds{1}_{\mathbb{R}^d \setminus \{0\}}$. 
\begin{align*}
      X_{i+1}^{j}(\theta) = & X_{i}^{j}(\theta) + \bar{b}(t_i,X_i^{j}(\theta),Y_{i}^{j}(\theta)) \Delta t_{i}  + \sigma  (t_i,X_i^j(\theta))   \Delta W_i^j 
      + \sum_{l = 1}^{\Delta N^{j}_{i}} \beta (t_i, X_i^j(\theta), \Delta J_l^{i,j}) \nonumber \\ 
\end{align*}
}
\For{$k = 0, \dots, A$ }{Sample $\Delta \bar N^{k}_{i}$ from a Poisson distribution $\mathcal{P}\left (\lambda \Delta t_{i} \right)$ and sample each element of the jumps sequence $(\Delta \bar J_l^{i,k})_{l = 1, \cdots, \Delta  \bar N^{k}_{i}}$ from the distribution  $\frac{\nu(de)}{\lambda} \mathds{1}_{\mathbb{R}^d \setminus \{0\}}$.}
\For{$j = 0, \dots, B$ }%{(carried out in parallel to estimate the compensator)}{ 
{
\begin{align*}
Y_{i+1}^{j}(\theta) =&  Y_{i}^{j}(\theta) - f(t_i,X_{i}^{j}(\theta),Y_{i}^{j}(\theta))\Delta t_{i}  +  \Zc^{\theta_{1}}(t_i,X_{i}^{j}(\theta)) \Delta W_i^j \\ 
      & +  \Wc^{\theta_{2}}(t_i, X_i^j(\theta), \sum_{l = 1}^{\Delta N^{j}_{i}} \tilde{\beta}_i^j(\Delta J^{i,j}_l))   - \frac{1}{A} \sum_{k=1}^A \left[  \Wc^{\theta_2}(t_i, X_i^j(\theta), \sum_{l = 1}^{\Delta \bar N^{k}_{i}} \tilde{\beta}_i^j(\Delta \bar J^{i,k}_l)) \right],\\
      %Y_{i+1}^{j}(\theta) =& Y_{i+1}^{j}(\theta) - \frac{1}{B} \sum_{k=1}^B \left[  \Wc^{\theta_2}(t_i, X_i^k(\theta),  \sum_{l = 1}^{\Delta N^{k}_{i}}  \Delta J^{i,k}_l) \right]
\end{align*}
}
}
$\phi(\theta) =  \frac{1}{B} \sum_{j = 1}^{B} \big |Y_{M}^{j}(\theta) - g(X_{M}^{j}(\theta)) \big|^{2}.$

$\theta = \theta - r_{m} \nabla  \phi(\theta)$}
 \caption{{\bf Global solver} \label{alg:globalDirac}}
\end{algorithm}
\clearpage
\textbf{2. SumLocal solver}. The second algorithm we develop  extends the one introduced in \cite{hure2020deep}. The setting with jumps is more involved and we propose two variants of this algorithm to deal with the jumps part of the backward component.
The $Y$ component is approximated by a neural network $\Uc^{\theta}$ and 
the two variants of the algorithm can be written as follows: 
\begin{itemize}
\item[(i)] We can directly use the system \eqref{eqn:discrete1} giving the \textit{SumLocal1} solver:
\begin{align*}
\Uc^\theta(t_{i+1}, X_{i+1}^{\pi}) \approx \quad & \Uc^\theta(t_{i}, X_i^{\pi}) -f(t_i,  X_i^{\pi}, \Uc^\theta(t_{i}, X_i^{\pi})) \Delta t_i + \Zc^\theta(t_i, X_i^{\pi}) \Delta W_i  +\\&
  \Uc^\theta(t_{i}, X_i^{\pi} +  \sum_{p=1}^{\Delta N_{i}} \beta(t_i, X_i^{\pi}, \Delta J_p^{i})) -
  \mathbb{E} \left[  \Uc^\theta(t_{i}, X_i^{\pi} +  \sum_{p=1}^{\Delta N_{i}} \beta(t_i, X_i^{\pi}, \Delta J_p^{i})) \Big |\mathcal{F}_{t_i} \right].
\end{align*}
    \item[(ii)] Or, as in the \textit{Global solver}, we can use a network  $\Wc^{\theta}$ for the jump part, which gives the \textit{SumLocal2} solver :
    \begin{align*}
\Uc^\theta(t_{i+1}, X_{i+1}^{\pi}) \approx \quad&  \Uc^\theta(t_{i}, X_{i}^{\pi})-f(t_i,  X_i^{\pi}, \Uc^\theta(t_{i}, X_i^{\pi})) \Delta t_i + \Zc^\theta(t_i, X_i^{\pi}) \Delta W_i  +\\&
   \Wc^\theta(t_i, X_i^{\pi} , \sum_{l = 1}^{\Delta N_{i}} \tilde{\beta}_i(\Delta J^{i}_l))    - 
  \mathbb{E} \left[
   \Wc^\theta(t_i, X_i^{\pi},  \sum_{l = 1}^{\Delta N_{i}}  \tilde{\beta}_i(\Delta J^{i}_l)) \Big |\mathcal{F}_{t_i} \right].
\end{align*}

\end{itemize}
Let $ \theta = (\theta_0, \theta_1, \theta_2)$ where $\theta_0 \in \R^{N^L_{d+1,m,kd}}$ are the parameters of the network function $\Zc^{\theta_0}$, $\theta_1 \in \R^{N^L_{1+2d,m,k}}$ are the parameters of the network function  $\Wc^{\theta_1}$, and  $\theta_2 \in \R^{N^L_{d+1,m,k}}$ are the parameters of the network function  $\Uc^{\theta_2}$.
We  detail the \textit{SumLocal2 solver} in  Algorithm \ref{alg:globalSumLocErrorDirac}.\\

\textbf{3. SumMultiStep solver.} The third algorithm represents a multistep version of the previous one, and extends the solver proposed in \cite{germain2020deep} to the jumps setting. It also has two versions, both based on the representation $\eqref{eqn:discrete1}$, for \textit{SumMultiStep1} we approximate the jumps part in the backward SDE as in $(i)$ above and for \textit{SumMultiStep2} we approximate the jumps part in the backward SDE as in $(ii)$ above.
Let $ \theta = (\theta_0, \theta_1, \theta_2)$ where $\theta_0 \in \R^{N^L_{d+1,m,kd}}$ are the parameters of the network function $\Zc^{\theta_0}$, $\theta_1 \in \R^{N^L_{1+2d,m,k}}$ are the parameters of the network function  $\Wc^{\theta_1}$, and  $\theta_2 \in \R^{N^L_{d+1,m,k}}$ are the parameters of the network function  $\Uc^{\theta_2}$.\
The \textit{SumMultiStep2  solver} is described in detail in Algorithm \ref{alg:globalDiracMultiStep}.\\

\begin{algorithm}[H]
\For{$ m = 0, \dots, K$}{
Set $\forall j \in [|1,B|]$ $X_{0}^{j}(\theta) = x_{0}$ \;
\For{$i= 0, \dots, M-1$}
{\For{$j = 1, \dots, B$}{ 
Sample $\Delta W_i^j$ from a Gaussian vector, sample $\Delta N^{j}_{i}$ from a Poisson distribution $\mathcal{P}\left (\lambda \Delta t_{i} \right)$ and sample each element of the jumps sequence $(\Delta J_l^{i,j})_{l = 1, \cdots, \Delta N^{j}_{i}}$ from the distribution  $\frac{\nu(de)}{\lambda} \mathds{1}_{\mathbb{R}^d \setminus \{0\}}$. 
\begin{align*}
      X_{i+1}^{j}(\theta) = & X_{i}^{j}(\theta) + \bar{b}(t_i,X_i^{j}(\theta),\Uc^{\theta_2}(t_i,X_{i}^{j}(\theta))) \Delta t_{i}  + \sigma (t_i,X_i^j(\theta))  \Delta W_i^j\\
      &+ \sum_{l = 1}^{\Delta N^{j}_{i}} \beta (t_i, X_i^j(\theta), \Delta J_l^{i,j}). 
\end{align*}}
\For{$k = 0, \dots, A$ }{Sample $\Delta \bar  N^{k}_{i}$ from a Poisson distribution $\mathcal{P}\left (\lambda \Delta t_{i} \right)$ and sample each element of the jumps sequence $(\Delta \bar J_l^{i,k})_{l = 1, \cdots, \Delta \bar  N^{k}_{i}}$ from the distribution  $\frac{\nu(de)}{\lambda} \mathds{1}_{\mathbb{R}^d \setminus \{0\}}$.}
}
{\small
\begin{align*}
    & \bullet \phi_{\text{local}}(\theta) = \sum_{i = 0}^{M-2} \left( \frac{1}{B} \sum_{j = 1}^{B}  \Big| \Uc^{\theta_2}(t_{i+1},X_{i+1}^{j}(\theta)) - \Uc^{\theta_2}(t_i,X_{i}^{j}(\theta)) +  f(t_i,X_{i}^{j}(\theta),\Uc^{\theta_2}(t_i,X_{i}^{j}(\theta))) \Delta t_{i} \right.\\& \left.  -  \Zc^{\theta_0}(t_i,X_{i}^{j}(\theta)) \Delta W_i^j  -   \Wc^{\theta_{1}}(t_i, X_i^j(\theta), \sum_{l = 1}^{\Delta N^{j}_{i}} \tilde{\beta}_i^j(\Delta J^{i,j}_l))  + \frac{1}{A} \sum_{k=1}^A \left[  \Wc^{\theta_2}(t_i, X_i^j(\theta), \sum_{l = 1}^{\Delta \bar N^{k}_{i}} \tilde{\beta}_i^j(\Delta \bar J^{i,k}_l)) \right]  \Big|^2 \right) \\
 & \bullet \phi_{\text{final}}(\theta) = \frac{1}{B} \sum_{j = 1}^{B} \bigg | g(X_{M}^{j}(\theta)) - \Uc^{\theta_2}(t_{M-1},X_{M-1}^{j}(\theta)) +\\
 &  f(t_{M-1},X_{M-1}^{j}(\theta),\Uc^{\theta_2}(t_{M-1},X_{M-1}^{j}(\theta))) \Delta t_{M-1}  -  \Zc^{\theta_0}(t_{M-1},X_{M-1}^{j}(\theta)) \Delta W_{M-1}  \\ & -     \Wc^{\theta_{1}}(t_{M-1}, X_{M-1}^{j}(\theta), \sum_{l = 1}^{\Delta N^{j}_{M-1}} \tilde{\beta}_i(\Delta J^{M-1,j}_l))  + \frac{1}{A} \sum_{k=1}^A \left[  \Wc^{\theta_2}(t_{M-1}, X_{M-1}^j(\theta), \sum_{l = 1}^{\Delta \bar  N^{k}_{M-1}} \tilde{\beta}_{M-1}^j(\Delta \bar J^{M-1,k}_l)) \right]  \bigg|^{2}\\
\end{align*}}
$\phi(\theta) = \phi_{\text{local}}(\theta) + \phi_{\text{final}}(\theta)$\\
$\theta = \theta - r_{m} \nabla  \phi(\theta)$}
      \caption{ \textbf{SumLocal  solver} (SumLocal2 variant). \label{alg:globalSumLocErrorDirac}}
\end{algorithm}

\begin{algorithm}[H]
\For{$ m = 0, \dots, K$}{
Set $\forall j \in [|1,B|]$ $X_{0}^{j}(\theta) = x_{0}$ \;
\For{$i= 0, \dots, M-1$}
{
\For{$j = 1, \dots, B$}{ 
\begin{flalign*}
 \psi_{i}^{j}(\theta)= \Uc^{\theta_2}(t_{i},X_{i}^{j}(\theta)) 
\end{flalign*}
Sample $\Delta W_i^j$ from a Gaussian vector, sample $\Delta N^{j}_{i}$ from a Poisson distribution $\mathcal{P}\left (\lambda \Delta t_{i} \right)$ and sample each element of the jumps sequence $(\Delta J_l^{i,j})_{l = 1, \cdots, \Delta N^{j}_{i}}$ from the distribution  $\frac{\nu(de)}{\lambda} \mathds{1}_{\mathbb{R}^d \setminus \{0\}}$. 
}
\For{$s = 0, \dots, A$ }{Sample $\Delta \bar N^{s}_{i}$ from a Poisson distribution $\mathcal{P}\left (\lambda \Delta t_{i} \right)$ and sample each element of the jumps sequence $(\Delta \bar J_l^{i,s})_{l = 1, \cdots, \Delta \bar N ^{s}_{i}}$ from the distribution  $\frac{\nu(de)}{\lambda} \mathds{1}_{\mathbb{R}^d \setminus \{0\}}$.}
\For{$k = 0, \dots, i$}{
\For{$j = 1, \dots, B$}{ 
\begin{align*}
 \psi_{k}^{j}(\theta)& = \psi_{k}^{j}(\theta) - f(t_i,X_{i}^{j}(\theta),\Uc^{\theta_2}(t_i,X_{i}^{j}(\theta)))\Delta t_{i}  + \Zc^{\theta_0}(t_i,X_{i}^{j}(\theta)) \Delta W_i^j \\&
      +\Wc^{\theta_{1}}(t_i, X_i^{j}(\theta), \sum_{l = 1}^{\Delta N_{i}^{j}} \tilde{\beta}_i(\Delta J^{i,j}_l))  - \frac{1}{A} \sum_{s=1}^A \left[  \Wc^{\theta_2}(t_i, X_i^j(\theta), \sum_{l = 1}^{\Delta \bar  N^{s}_{i}} \tilde{\beta}_i^j(\Delta \bar J^{i,s}_l)) \right].
\end{align*}
}
}
\For{$j = 1, \dots, B$} { %(carried out in parallel to estimate the compensator)}{ 
\begin{align*}
      X_{i+1}^{j}(\theta) &=  X_{i}^{j}(\theta) + \bar{b}(t_i,X_i^{j}(\theta),\Uc^{\theta_2}(t_i,X_{i}^{j}(\theta))) \Delta t_{i}  + \sigma (t_i,X_i^j(\theta))  \Delta W_i^j\\
      &+ \sum_{l = 1}^{\Delta N^{j}_{i}} \beta (t_i, X_i^j(\theta), \Delta J_l^{i,j}).   
\end{align*}
}}
$\phi(\theta) = \sum_{i = 0}^{M-1} \left( \frac{1}{B} \sum_{j = 1}^{B} \big | \psi_{i}^{j}(\theta) - g(X_{M}^{j}(\theta)) \big |^{2} \right)$ \\   
$\theta = \theta - r_{m} \nabla  \phi(\theta)$}
\caption{The \textbf{SumMultiStep} solver (SumMultiStep2 variant). \label{alg:globalDiracMultiStep}}
\end{algorithm}
\clearpage
\vspace{2mm}
\subsubsection{Second class  of deep-learning algorithms.}
In this second part, we introduce the deep-learning algorithms based on the regression methods.

\noindent The following algorithms exploit the fact the driver does not depend on $Z$ and $U$, thus 
use a single network $\Uc^{\theta}$ to approximate the $Y$ component of the solution.
By conditionning the backward component in \eqref{eqn:discrete1}, we obtain \begin{align*}
    \Uc^{\theta}(t_{i}, X_{i}^{\pi})  \approx \E \left [ \Uc^{\theta}(t_{i+1}, X_{i+1}^{\pi}) + f(t_i,X_{i}^{\pi} ,\Uc^{\theta}(t_{i}, X_{i}^{\pi}))\Delta t_{i} |\mathcal{F}_{t_i}  \right].
\end{align*}
\textbf{1. SumLocalReg solver.} The first algorithm based on the regression methods is the \textit{SumLocalReg  solver}, which is a neural network version of the algorithms developed in \cite{gobet2005regression}, \cite{lemor2006rate}. It  is described in detail in Algorithm \ref{alg:globalSumLocErrorReg}.\\
\begin{algorithm}[H]
\For{$ m = 0, \dots, K$}{
Set $\forall j \in [|1,B|]$ $X_{0}^{j}(\theta) = x_{0}$ \;
\For{$i= 0, \dots, M-1$}
{\For{$j = 1, \dots, B$}
{%(carried out in parallel to estimate the compensator)}{ 
Sample $\Delta W_i^j$ from a Gaussian vector, sample $\Delta N^{j}_{i}$ from a Poisson distribution $\mathcal{P}\left (\lambda \Delta t_{i} \right)$ and sample each element of the jumps sequence $(\Delta J_l^{i,j})_{l = 1, \cdots, \Delta N^{j}_{i}}$ from the distribution  $\frac{\nu(de)}{\lambda} \mathds{1}_{\mathbb{R}^d \setminus \{0\}}$. 
\begin{align*}
      X_{i+1}^{j}(\theta) = & X_{i}^{j}(\theta) + \bar{b}(t_i,X_i^{j}(\theta),\Uc^{\theta}(t_{i}, X_{i}^{j}(\theta))) \Delta t_{i}  + \sigma (t_i,X_i^j(\theta))  \Delta W_i^j\\
      &+ \sum_{l = 1}^{\Delta N^{j}_{i}} \beta (t_i, X_i^j(\theta), \Delta J_l^{i,j}). 
\end{align*}}}
\begin{align*}
    &\phi_{\text{local}}(\theta) = \sum_{i = 0}^{M-2} \left( \frac{1}{B} \sum_{j = 1}^{B}  \Big| \Uc^{\theta}(t_{i+1}, X_{i+1}^{j}(\theta)) - \Uc^{\theta}(t_{i}, X_{i}^{j}(\theta))  +  f(t_i,X_{i}^{j}(\theta),\Uc^{\theta}(t_{i}, X_{i}^{j}(\theta))) \Delta t_{i}  \Big|^2 \right) \\
    &\phi_{\text{final}}(\theta) = \frac{1}{B} \sum_{j = 1}^{B}  \Big| g(X_{M}^{j}(\theta)) -  \Uc^{\theta}(t_{M-1}, X_{M-1}^{j}(\theta)) + 
    f(t_{M-1},X_{M-1}^{j}(\theta),\Uc^{\theta}(t_{M-1}, X_{M-1}^{j}(\theta)))\Delta t_{M-1}) \Big|^{2} \\
    &\phi(\theta) = \phi_{\text{local}}(\theta) + \phi_{\text{final}}(\theta)
\end{align*}
$\theta = \theta - r_{m} \nabla  \phi(\theta)$}
      \caption{ The \textbf{SumLocalReg solver}.\label{alg:globalSumLocErrorReg}}
\end{algorithm} 
\textbf{2. SumMultiStepReg solver.} The second  algorithm  is the \textit{SumMultiStepReg solver}, which is a multistep version of the previous one, in the same spirit as in \cite{bender2007forward}. It is described in detail in Algorithm \ref{alg:globalDiracMultiStepReg}.

\begin{algorithm}[H]
\For{$ m = 0, \dots, K$}{
Set $\forall j \in [|1,B|]$ $X_{0}^{j}(\theta) = x_{0}$\;
\For{$i= 0, \dots, M-1$}
{
\For{$j = 1, \dots, B$}{ 
\begin{align*}
 \psi_{i}^{j}(\theta) = \Uc^{\theta}(t_{i}, X_{i}^{j}(\theta))   
\end{align*}

Sample $\Delta W_i^j$ from a Gaussian vector, sample $\Delta N^{j}_{i}$ from a Poisson distribution $\mathcal{P}\left (\lambda \Delta t_{i} \right)$ and sample each element of the jumps sequence $(\Delta J_l^{i,j})_{l = 1, \cdots, \Delta N^{j}_{i}}$ from the distribution  $\frac{\nu(de)}{\lambda} \mathds{1}_{\mathbb{R}^d \setminus \{0\}}$. 
}
\For{$k = 0, \dots, i$}{
\For{$j = 1, \dots, B$}{ 
$
 \psi_{k}^{j}(\theta)= \psi_{k}^{j}(\theta) - f(t_i,X_{i}^{j}(\theta),\Uc^{\theta}(t_{i}, X_{i}^{j}(\theta)))\Delta t_{i}
$
}}
\For{$j = 1, \dots, B$}{ 
\begin{align*}
      X_{i+1}^{j}(\theta) &=  X_{i}^{j}(\theta) + \bar{b}(t_i,X_i^{j}(\theta),\Uc^{\theta}(t_{i}, X_{i}^{j}(\theta))) \Delta t_{i}  + \sigma  (t_i,X_i^j(\theta)) \Delta W_i^j \\&+ \sum_{l = 1}^{\Delta N^{j}_{i}} \beta (t_i, X_i^j(\theta), \Delta J_l^{i,j}).  
\end{align*}
}}
$\phi(\theta) = \sum_{i = 0}^{M-1} \left( \frac{1}{B} \sum_{j = 1}^{B} \big | \psi_{i}^{j}(\theta) - g(X_{M}^{j}(\theta)) \big |^{2} \right)$ \\   
$\theta = \theta - r_{m} \nabla  \phi(\theta)$}
\caption{The \textbf{SumMultiStepReg} solver  \label{alg:globalDiracMultiStepReg}}
\end{algorithm}

\begin{remark}
As proposed in \cite{Gnoatto22} in the case of decoupled FBSDEs, another method to estimate the compensator is to consider an additional neural network function $\mathcal{C}^{\theta_3}$ approximating the compensator $t,x \to \int_{\mathbb{R}^d \setminus \{0\}} \left(u(t, x + e) - u(t, x)\right)\nu(de)$ by adding an additional penalty term to the original loss function. Hence, the conditional expectation $\mathbb{E} \left[
   \Wc^{\theta_2}(t_i, X_i,  \sum_{l = 1}^{\Delta N_{i}}  \tilde{\beta}_i(\Delta J^{i}_l)) \big \lvert \mathcal{F}_{t_i} \right]$ is estimated by $\mathcal{C}^{\theta_3}(t_i, X_i) \Delta t_i$ at each time step by optimizing the following penalty function: 
\begin{equation} \label{eq:condApprox}
    \sum_{i = 0}^{M -1} \left[ \left \lvert \mathcal{W}^{\theta_2}(t_i,X_i(\theta), \sum_{l = 1}^{\Delta N_{i}} \tilde{\beta}_i(\Delta J^{i}_l)) - \mathcal{C}^{\theta_3}(t_i, X_i(\theta)) \Delta t_i \right\lvert^2  \right].
\end{equation}
We tested this approximation method in the coupled case for the Global method and both local methods, and the algorithms lacked accuracy due to the presence of an additional neural network and an additional term in the loss function. Our approach is based on the approximation of the conditional expectation as in \eqref{eq:condApprox} to directly estimate the backward component $Y$ in the \textit{SumLocalReg} and \textit{SumMultiStepReg} algorithms, which yields better results. 
%The idea of approximating the conditional expectation \eqref{eq:condApprox} applied to the regression algorithm developed in \cite{han2018solving} was also implemented in \cite{Gnoatto22} for the deep BSDE and seemed to work well in the decoupled case.
\end{remark}

\begin{remark}
The above algorithms can be used to handle the general case of jumps with infinite activity,  after truncating the small jumps as in \cite{dumitrescu2021approximation, Gnoatto22}.
\end{remark}

\subsection{Numerical tests for option pricing}

In this subsection, we aim to assess the performance of the deep learning algorithms discussed in the preceding section in the context  of pricing European options in three different financial models: the Black-Scholes (BS) model (without jumps), the Merton (MJ) model (with jumps with finite activity), and the Variance Gamma (VG) model (with jumps with \textit{infinite} activity). Indeed, we can adapt the algorithms presented in the finite-activity setting to the pricing of European-options under the exponential Variance-Gamma model. 

We first set the hyper-parameters for the \textit{Global} solver and both variants of the \textit{SumLocal} and \textit{SumMultiStep} solvers where NbTraining corresponds to the number of gradient iterations of the Adam stochastic gradient descent algorithm \cite{kingma2014adam}.

\begin{table}[H]
\begin{center}
\begin{tabular}{| c | c |}
\hline
Parameter & value\\
\hline
 $m$ & 21\\
 $L$ & 2\\
 NbTraining & 12000\\
\hline
\end{tabular}
\quad
\begin{tabular}{| c | c |}
\hline
Parameter & value\\
\hline
 A & 5000\\
 B & 10\\
 $\sigma_{a}$ & tanh\\
\hline
\end{tabular}
\quad
\begin{tabular}{| c | c |}
\hline
\end{tabular}
\caption{Hyper-parameters for the first class of deep-learning algorithms}
\end{center}
\end{table}

Furthermore, the specific parameters of the regression methods (since the compensator is not computed) are:
\begin{table}[H]
\begin{center}
\begin{tabular}{| c | c |}
\hline
Parameter & value\\
\hline
 $m$ & 21\\
 $L$ & 2\\
 NbTraining & 12000\\
\hline
\end{tabular}
\quad
\begin{tabular}{| c | c |}
\hline
Parameter & value\\
\hline
 B & 10 000\\
 $\sigma_{a}$ & tanh\\
\hline
\end{tabular}
\quad
\begin{tabular}{| c | c |}
\hline
\end{tabular}
\end{center}
\caption{Hyper-parameters for the second class of deep-learning algorithms}
\end{table}
The algorithms are implemented in Python with \textit{Tensorflow} library. Each numerical experiment is conducted using GPU Tesla T4-PCIE-16GB. The code for the numerical experiments of the pricing and Mean Field Game (MFG) models can be accessed at the following URL: \url{https://github.com/ZakariaBensaid/DeepFBSDEJSolvers}.

 In the Black-Scholes and Merton models, we compare the results we get by implementing our deep learning algorithms with those obtained by using the well-known \textit{closed formula} of the solutions of the PDE, respectively PIDE. In the Variance Gamma model, we compare our results with those obtained by using the \textit{inverse Fourier} method computed with the Fast Fourier Transform algorithm.
 
 \subsubsection{The models}
 We present below the three models.\\
 
\noindent \textbf{The Black Scholes model} \textit{(No jumps)}
The BS model proposes to model the underlying asset $S_t$ under the risk neutral probability measure $\mathbb{Q}$ following a geometric Brownian motion with the following dynamics: $S_t = s\exp((r-\frac{\sigma^2}{2})t+\sigma W_t)$.  The problem of pricing an European call option in the BS model translates to the following FBSDE:
\begin{align}
    \begin{cases}
        dS_t = S_t(rdt+\sigma dW_t), \\
        -dY_{t} = - rY_{t}dt - Z_{t} dW_{t},\\
        S_0 = s, \quad Y_T = (S_T - K)^+ .
    \end{cases}
    t \in [0, T],
\end{align}
 where $K$ is the strike price. More precisely, the price of the European option at time $t$ is given by $Y_t$. Furthermore,  it is known that there exists a function $\bar{u}$ such that $Y_t=\bar{u}(t,S_t)$, where the function $\bar{u}$ solves a specific PDE.
 
To test the performance of the algorithms in a coupled setting, we propose below a forward-backward system for which the forward component has an additional term coupled to the backward component. More precisely, we consider the following system
\begin{align}\label{sysC}
    \begin{cases}
        dX_t = X_t(rdt+\sigma dW_t) + a |Y_t - \bar{u}(t, X_{t}) |dt, \\
        -dY_{t} = - rY_{t}dt - Z_{t} dW_{t},\\
        X_0 = S_0, \quad Y_T = (X_T - K)^+ .
    \end{cases}
    t \in [0, T],
\end{align} 
\noindent where $\bar{u}$ is the analytical solution of the PDE in the decoupled case. In the case of a small time maturity, Theorem \ref{thm:smallmat} guarantees that the system \eqref{sysC} admits an unique solution for which the backward component $Y$ provides the price of the european call option. This applies for all the models below.  \\

 \textit{Model parameters.} For the numerical implementation, we set $T = 1$ , $M = 50$ steps, the interest rate $r = 0.1$, the diffusion volatility $ \sigma = 0.3$, the strike price $K = 0.9$, the spot price $S_0 = 1$, and the coupling linearity coefficient when non-null $ a = 0.1$ .
\\

\noindent \textbf{Merton model} \textit{(Jumps with finite activity)}
Merton's \cite{merton_jump} approach proposes to ignore risk premia for jumps, this assumption leading to a specific choice for pricing and hedging. To describe the model, we assume that the underlying asset $S_t$ under the risk neutral probability measure $\mathbb{Q}$ follows the dynamics $S_t=S_0\exp((r-\sigma^2/2-m) t+\sigma W_t+\sum_{i=1}^{N_t}Y_i)$, where $N_t, Y_{i}$ are independent from $W_t$ and $N_t$ is a Poisson process with intensity $\lambda t$. The random variables $Y_i$ are i.i.d. and follow a $\mathcal{N}(\alpha, \xi^2)$ distribution. The constant $m$ is chosen  such that the process $\Tilde{S}_t=S_t e^{-rt}$ is a martingale under $\mathbb{Q}$ and is given by $m:=\lambda \mathbb{E}[e^{Y_i}-1]$. 

As above, under appropriate assumptions on the coefficients, we can express the problem of pricing an European call option in the Merton model in terms of the following coupled FBSDE: 
\begin{align}
    \begin{cases}
        dX_t = X_{t^-}(rdt+\sigma dW_t + \int_{\mathbb{R}^{\star}} (e^e - 1) \Tilde{\mathcal{J}}(dt,de)) + a |Y_t - \bar{u}(t, X_{t}) |dt, \\
        -dY_{t} = - rY_{t}dt - Z_{t} dW_{t} - \int_{\mathbb{R}^\star}U_{t}(e) \Tilde{\mathcal{J}}(dt,de),\\
        X_{0} = S_{0}; \quad Y_T = (X_T - K)^+ .\\
    \end{cases}
    t \in [0, T],
\end{align}
\noindent where $\bar{u}$ is the analytical solution of the partial integro differential equation associated to the  decoupled FBSDE and $\Tilde{\mathcal{J}}(dt,de)$ is the compensated jump measure associated with the compound Poisson process $\sum_{i=1}^{N_t} Y_{i}$ with intensity measure $\nu(de)$, where $\nu$ is given by : $$\nu(de) = \frac{\lambda}{\xi \sqrt{2 \pi}} \exp \left (- \frac{( e - \alpha )^2}{2 \xi^2} \right)de.$$

 \textit{Model parameters.} For this example, we set $T = 1$ , $M = 50$ steps, the interest rate $r = 0.1$, the diffusion volatility $ \sigma = 0.3$, the jumps intensity $\lambda = 3$, the parameters of the jumps distribution $\alpha = 0$ and $\xi = 0.2$, the strike price $K = 0.9$, the initial condition $X_0 = 1$, and the coupling linearity coefficient when non-null $ a = 0.1$ .
\\

    It can be observed that the algorithms introduced in the previous section in the case of finite activity jumps  can be used to compute the solution of coupled FBSDEs in some particular case of infinite activity jumps. In particular,  in the case when $\beta(t,x,e)=\gamma(x)\cdot e $, the jump process $J_t$ has finite variation and its jumps between two consecutive points on the grid can be simulated. In particular, this can be implemented for jump models based on the Brownian subordination, such as the Gamma or Variance-Gamma processes. We present below the results we obtained for the Variance-Gamma model.\\
    
%Nevertheless, we observed that the models based on \textit{Brownian subordination} such as the Variance Gamma model can be adapted to our algorithms. The idea is to sample the jumps part $J_t$ by simulating the time changed Brownian motion rather than sampling each element of the jumps sequence from the Lévy measure. 

\noindent \textbf{Variance Gamma model} \textit{(Jumps with infinite activity)}
The Variance-Gamma process is a L\'evy process with \textit{infinite activity jumps}, where the jumps part $J_t$ have a Variance-gamma law $VG(\bar{\sigma}, \kappa, \theta)$ (see e.g. \cite{Madan1990}). 
Its characteristic function is 
\begin{align*}
\mathbb{E}[e^{i u J_t}]= \left(1-iu \theta \kappa+ \frac{1}{2} \bar{\sigma}^2 \kappa u^2\right)^{-\frac{t}{\kappa}}.
\end{align*}

\noindent The variance-gamma process can be characterized as a time changed Brownian motion with drift, i.e.
$$J_{t} = \theta T_t + \bar{\sigma} W_{T_t},$$
where $W_t$ is a standard Brownian motion, $T_t \sim \Gamma(t, \kappa t)$ and $\theta$, $\kappa$, $\bar{\sigma}$ are given constants. 
The intensity measure of a Variance-Gamma process is given by
$$\nu(de) = \frac{\exp\left({\frac{\theta e}{\bar{\sigma}^2}}\right)}{\kappa|e|} \exp 
 \left( - \frac{\sqrt{\frac{2}{\kappa} + \frac{\theta^2}{\bar{\sigma}}^2}}{\bar{\sigma}} |e|\right) de$$
%In the variance gamma model \cite{Madan1990}, the underlying asset price follows a Lévy process with infinite activity jumps, which can be represented by a process with a CGMY distribution (which is an extension of \cite{Madan1990} developed in \cite{Madan1998} ). The CGMY process is defined as the sum of C, G, M and Y distributed independent increments, where C, G, M and Y are the four parameters of the model. To simulate the model, we use the Brownian subordination technique: Fistly, let $X_t = \theta t + \bar{\sigma} W_t$ represent a Brownian motion with drift, where $W_t$ is a standard Brownian motion and $\theta, \bar{\sigma}$ are constants. If we substitute the time variable $t$ with the gamma random variable $T_t \sim \Gamma(t, \kappa t)$, then we obtain the variance gamma process:

%$$X_{T_t} = \theta T_t + \bar{\sigma} W_{T_t}$$

\noindent Under the risk neutral probability measure $\mathbb{Q}$, we assume that the underlying asset $S_t$ follows the dynamics 
\begin{align*}
S_t=S_0\exp((r+\omega)t+J_t),
\end{align*}
where 
$$\omega=\kappa^{-1}\log \left(1-\frac{1}{2}\bar{\sigma}^2 \kappa-\theta \kappa\right).$$

Similarly to the BS and MJ models, we consider below the following coupled FBSDE system which, under appropriate assumptions on the coefficients, provides the  price of an European call option in the VG model: 

\begin{align}
    \begin{cases}
            dX_t = X_{t^-}(rdt+ \int_{\mathbb{R}^{\star}} (e^e - 1) \Tilde{\mathcal{J}}(dt,de)) + a |Y_t - \bar{u}(t, X_{t}) |dt, \\
         -dY_{t} = - rY_{t} dt - \int_{\mathbb{R}^\star}U_{t}(e) \Tilde{\mathcal{J}}(dt,de), \\
         X_{0} = S_0; \quad Y_T = (X_T - K)^+ ,\\
    \end{cases}
    t \in [0, T],
\end{align}
where $\Tilde{\mathcal{J}}(dt,de)$ is the compensated jump measure associated with the variance Gamma process $J_t$, and $\bar{u}$ is the analytical solution of the PIDE in the decoupled case.

 \textit{Model parameters.} For this example, we set $T = 1$ , $M = 30$ steps, the interest rate $r = 0.1$, the time-scaled Brownian motion drift and volatility $\theta = - 0.1$ and $ \bar{\sigma} = 0.2$, the variance of the Gamma process $\kappa = 0.1$, the strike price $K = 0.9$, the initial condition $X_0 = 1$, and the coupling linearity coefficient when non-null $ a = 0.1$ .
 
% \begin{figure}[H] %[htbp]
%     \centering
%     \begin{subfigure}[b]{0.5\textwidth}
%         \centering
%         \subfloat[Coupled case ($ a \neq 0$)]{\includegraphics[width=\linewidth]{OptionPricing/VGCoupled.png}}
%     \end{subfigure}%
%     \begin{subfigure}[b]{0.5\textwidth}
%         \centering
%         \subfloat[Decoupled case ($ a = 0$)]{\includegraphics[width=\linewidth]{OptionPricing/VGDecoupled.png}}
%     \end{subfigure}\\[1ex]
% Convergence of the 7 algorithms in the Variance Gamma model   
% \end{figure}
% \gr{ Comment results.}

%\subsection{Summary}
\subsubsection{Results}
On Figure \ref{fig:BS}, \ref{fig:Merton}, \ref{fig:VG}, we plot the convergence of the different algorithms for the BS model, the Merton model and the Variance Gamma model, for $a=0$ (the decoupled case) and $a$ different from $0$ (the coupled case). Hence, we plot the evolution of $Y_0$ through 100 epochs for BS and 120 epochs for MJ and VG. Notice that 100 gradient descents are performed between 2 epochs for BS, MJ and VG. 
\begin{figure}[H]%[htbp]
    \centering
    \begin{subfigure}[b]{0.5\textwidth}
        \centering
        \subfloat[Coupled case ($ a = 0.1$)]{\includegraphics[width=\linewidth]{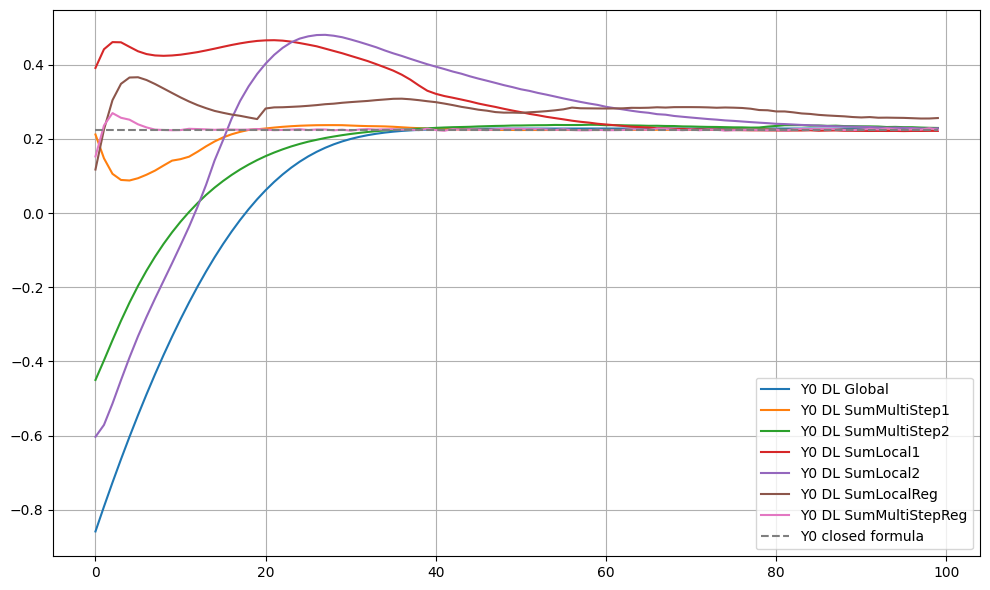}}
    \end{subfigure}%
    \begin{subfigure}[b]{0.5\textwidth}
        \centering
        \subfloat[Decoupled case ($ a = 0$)]{\includegraphics[width=\linewidth]{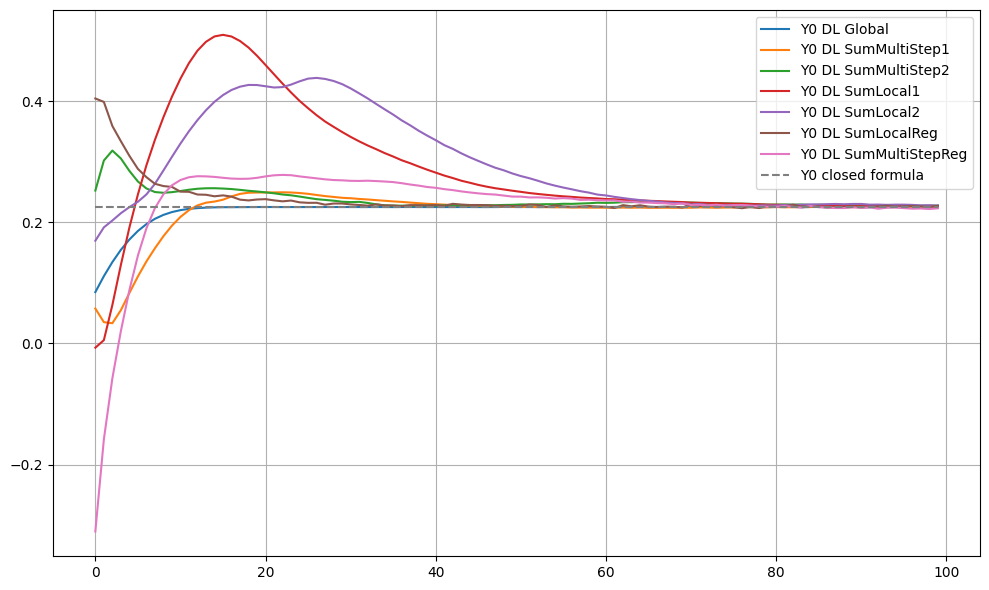}}
    \end{subfigure}\\[1ex]
\caption{Convergence of the 7 algorithms in the BS model \label{fig:BS}}
\end{figure}
Figure \ref{fig:BS} illustrates the convergence of the European call price $Y_0$ in the Black-Scholes model in both, the coupled and decoupled cases. This figure demonstrates that all methods converge smoothly in the decoupled case. In the coupled system, all methods also converge smoothly except for SumLocalReg, which stagnates between 0.255 and 0.256, instead of converging to the true value 0.225. 

\begin{figure}[H]%[htbp]
    \centering
    \begin{subfigure}[b]{0.5\textwidth}
        \centering
        \subfloat[Coupled case ($ a = 0.1 $)]{\includegraphics[width=\linewidth]{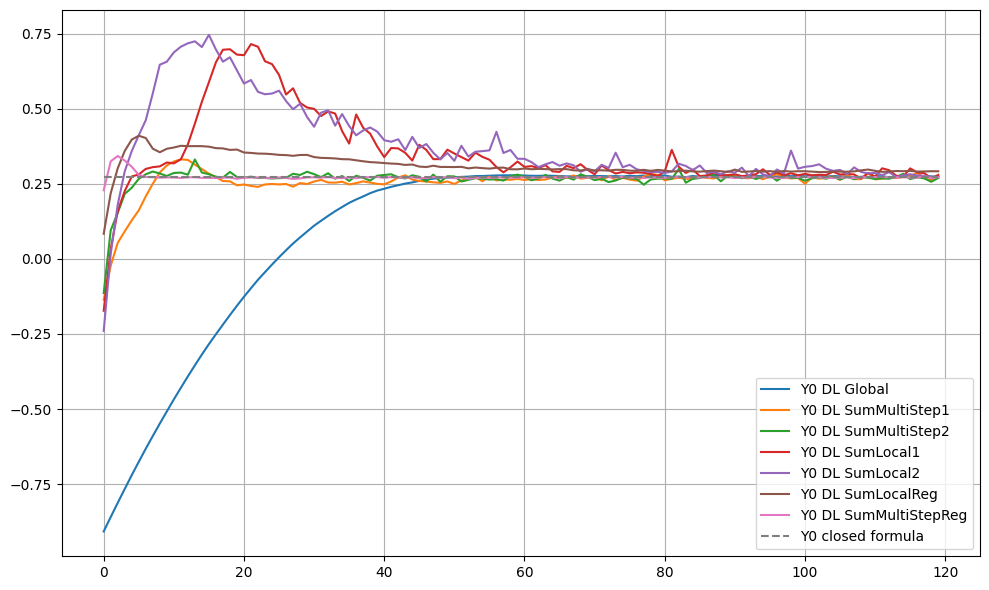}}
    \end{subfigure}%
    \begin{subfigure}[b]{0.5\textwidth}
        \centering
        \subfloat[Decoupled case ($ a = 0$)]{\includegraphics[width=\linewidth]{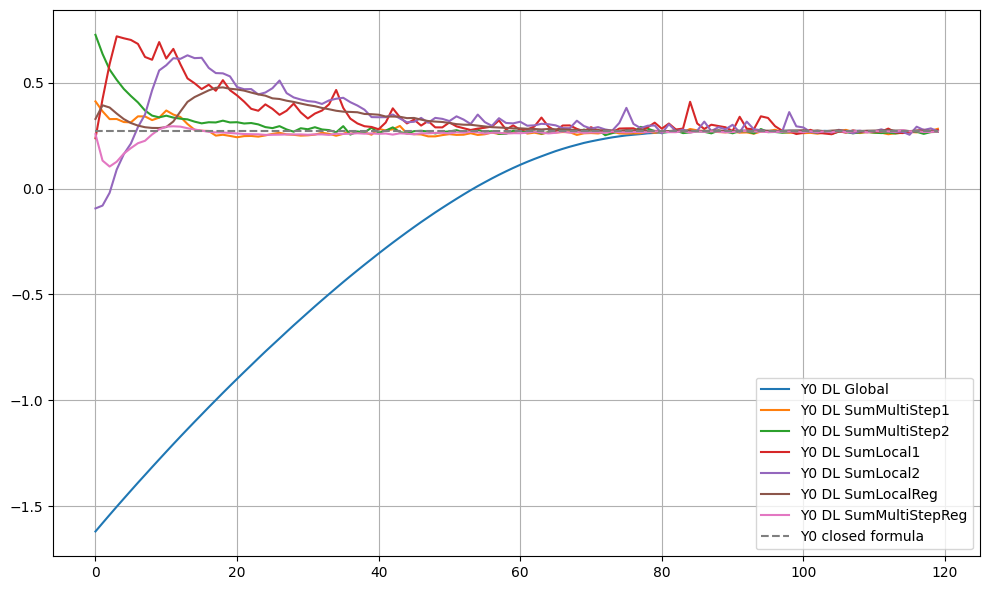}}
    \end{subfigure}\\[1ex]
\caption{Convergence of the 7 algorithms in the Merton model \label{fig:Merton}}
\end{figure}

Figure \ref{fig:Merton} illustrates the convergence of the value of the European call price $Y_0$ for the Merton model, in both the coupled and decoupled cases. This model allows us to test the performance of our algorithms in a setting that involves a jump diffusion model with finite activity. It can be observed that all the algorithms converge quickly, requiring only 80 epochs, except for SumLocalReg in the coupled case. Similar to its performance in the BS model, SumLocalReg is very unstable and far from the true value.

\begin{figure}[H] %[htbp]
    \centering
    \begin{subfigure}[b]{0.5\textwidth}
        \centering
        \subfloat[Coupled case ($ a =0.1$)]{\includegraphics[width=\linewidth]{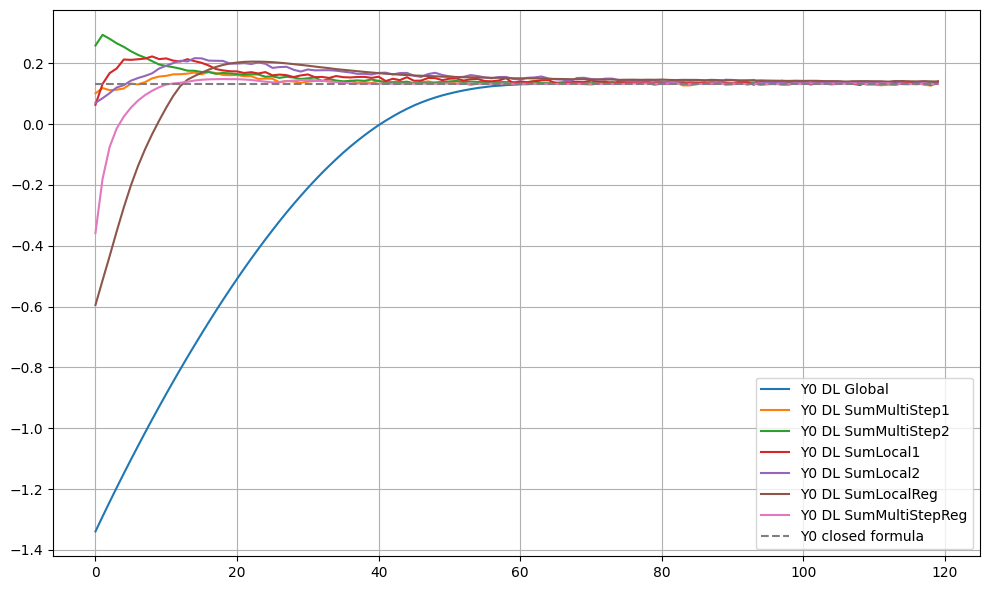}}
    \end{subfigure}%
    \begin{subfigure}[b]{0.5\textwidth}
        \centering
        \subfloat[Decoupled case ($ a = 0$)]{\includegraphics[width=\linewidth]{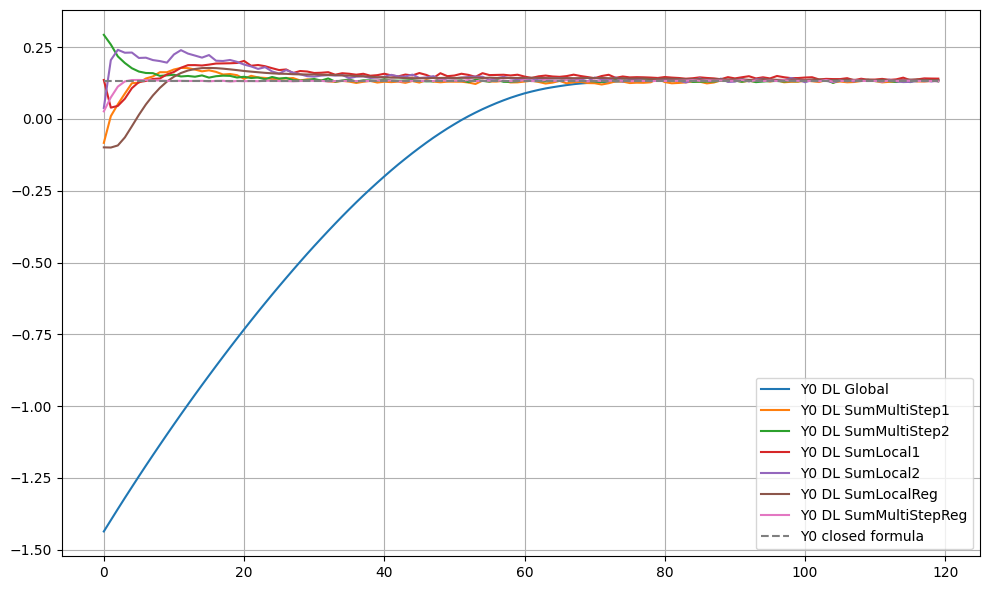}}
    \end{subfigure}\\[1ex]
\caption{ Convergence of the 7 algorithms in the Variance Gamma model  \label{fig:VG}} 
\end{figure}

Figure \ref{fig:VG} illustrates the convergence of the value of the European call price $Y_0$ for the Variance Gamma model, in both the coupled and decoupled cases. This model allows us to test the performance of our algorithms in a model with pure jumps with infinite activity. We observed that all the algorithms were consistent, stable and relatively quick except for the SumLocalReg algorithm.

To focus on the processing times intrinsic to the training process, we first present the computation times  and results for the Merton and Variance Gamma models with the parameter $a = 0$. This removes the additional computation time required for the analytical solution $u$ that is not part of the training process.
\begin{table}[H]
\centering

\begin{tabular}{@{}llllllll@{}}
\toprule
\multirow{2}{*}{Model} & \multicolumn{7}{c}{DL Methods} \\ \cmidrule(l){2-8} 
& Global & MultiStep1 & MultiStep2 & SumLocal1 & SumLocal2 & SumLocalReg & MultiStepReg \\ \midrule
MJ & 874s & 941s & 928s & 921s & 902s & 622s & 642s\\
VG & 634s & 687s & 692s & 673s & 670s & 666s & 702s\\ \bottomrule
\end{tabular}
\caption{Computation times in seconds for different DL methods after 12000 training steps}\label{tab:comp_times}
\end{table}
As shown in Table \ref{tab:comp_times}, the performance of the different deep learning methods for the Merton and Variance Gamma models was evaluated based on the computation time required for  12000 training steps. Overall, the results indicate that the MultiStepReg and SumLocalReg methods were the most time-efficient for the Merton model. Nonetheless, the discrepancies are not significant enough to base our choice on the computation time only. Thus, we present some accuracy and convergence results.\\

In Table \ref{tab:conva0}, we present the results obtained for $a=0$ after 12000 training steps.
\begin{table}[H]
\centering
\begin{tabular}{@{}llllllll@{}}
\toprule
\multirow{2}{*}{Model} & \multicolumn{7}{c}{DL Methods} \\ \cmidrule(l){2-8} 
& Global & MultiStep1 & MultiStep2 & SumLocal1 & SumLocal2 & SumLocalReg & MultiStepReg\\ \midrule
MJ &  \cellcolor{green!50} 0.271 & \cellcolor{green!50} 0.273 &  \cellcolor{red!50} 0.266 &\cellcolor{red!50} 0.276 & \cellcolor{green!50} 0.270 & \cellcolor{green!50} 0.272 & \cellcolor{green!50} 0.267\\
VG & \cellcolor{green!50} 0.133 & \cellcolor{green!50} 0.132 & \cellcolor{green!50} 0.137 & \cellcolor{red!50} 0.141 & \cellcolor{green!50} 0.130 & \cellcolor{green!50} 0.135 & \cellcolor{green!50} 0.132\\ \bottomrule
\end{tabular}
\caption{ $Y_0$ for $a=0$ and different DL methods.
The reference value for the Merton and Variance Gamma models are $0.271$ respectively $0.133$. The green color corresponds to an error less than $4.10^{-3}$ and the red color to an error larger than $4.10^{-3}$.}
\label{tab:conva0}
\end{table}
In Table \ref{tab:conva01}, we present the results obtained for $a=0.1$ after 12000 training steps.

\begin{table}[H]
\centering
\begin{tabular}{@{}llllllll@{}}
\toprule
\multirow{2}{*}{Model} & \multicolumn{7}{c}{DL Methods} \\ \cmidrule(l){2-8} 
& Global & MultiStep1 & MultiStep2 & SumLocal1 & SumLocal2 & SumLocalReg & MultiStepReg\\ \midrule
MJ &  \cellcolor{green!50} 0.273 & \cellcolor{green!50} 0.274 & \cellcolor{green!50} 0.269 & \cellcolor{red!50} 0.280 & \cellcolor{green!50} 0.273 & \cellcolor{red!50} 0.292 & \cellcolor{green!50} 0.272\\
VG & \cellcolor{green!50} 0.133 & \cellcolor{green!50} 0.136 & \cellcolor{green!50} 0.135 & \cellcolor{red!50} 0.141 & \cellcolor{green!50} 0.135 & \cellcolor{red!50} 0.140 & \cellcolor{green!50} 0.132\\ \bottomrule
\end{tabular}
\caption{ $Y_0$ for $a=0.1$ and different DL methods.
The reference value for the Merton and Variance Gamma models are the same as in Table \ref{tab:conva0}.}
\label{tab:conva01}
\end{table}
\subsubsection{Conclusion}
The results above show that neural network methods can solve coupled FBSDEs with jumps with finite activity (or a particular class of jumps with infinite activity as explained above) issued from pricing models. 
After various benchmarks, we observe that 
\begin{enumerate}
    \item The \textit{Global} method is consistent, stable and relatively robust compared to the other methods when it comes to the calibration of the hyper-parameters (especially the learning rate).
    \item The first variant \textit{SumLocal1} of the local methods is not very accurate in the decoupled and coupled cases, whereas the regression version \textit{SumLocalReg} is accurate in the decoupled case and faces more difficulties in the coupled case. On the other hand, \textit{SumLocal2} performs well in the decoupled and coupled cases with the fine-tuned hyper-parameters which have an important impact on the accuracy of the local algorithms.
    \item All the MultiStep variants \textit{MultiStep1}, \textit{MultiStep2} and the regression version \textit{MultiStepReg}, perform very well conditionally on finding the adequate hyperparameters that depend on the parameters of the models.
    \item Finally, \textit{MultiStepReg} provides the best computation speed and a good accuracy without the need to estimate the compensator using Monte Carlo or additional networks which will add extra biases.
\end{enumerate}

%\begin{remark}[Numerical results in the multi-dimensional case]
%In the multi-dimensional case with $d = 100$ in the Merton model, both the \underline{Global} and \underline{SumMultiStep2} methods have demonstrated rapid convergence, providing satisfactory results. This is in contrast to the unidimensional case where \underline{SumMultiStep1} was superior. The \underline{Regression} version of \underline{SumMultiStep} also exhibited efficiency, completing in a relatively short time when choosing smaller batch sizes (since we do not estimate the compensator in this setting). However, in comparison with the unidimensional case, the perfomance of the \underline{SumLocal} and \underline{SumLocalReg} was suprisingly poor.

%It is also noteworthy that the Global solver converged in 1855 seconds with a higher precision and with 3800 iterations only, compared to 40,000 in \cite{Gnoatto22} in the Merton model with the same parameters. This underlines the efficiency of our Global method in this specific multi-dimensional context.
%\end{remark}

\section{Application to an MFG model with jumps for smart grids}\label{sec:MFG} 
In this section, we develop a generalized version of the model introduced in \cite{MFG_revised} (which is also related to the MFG models presented in \cite{alasseur2019extended, matoussi1}), which we solve numerically using the deep learning algorithms introduced in Section \ref{alg}. We consider an energy system with $n$ consumers who are linked by a Demand Side Management (DSM) contract, i.e. they agree to diminish, at random times, their aggregated power consumption by a predefined volume during a predefined duration. Their failure to deliver the service is penalised via the difference between the sum of the $n$ power consumptions and the contracted target. The jumps are supposed to come from a Cox process with a \textit{stochastic intensity process}, in contrast with \cite{MFG_revised} where the intensity is only assumed to be constant. From a modeling perspective, this generalization is important since it allows to capture the dependence of the intensity on e.g. the aggregated consumption, which implies that the jumps arrive with a higher rate when the demand is at its peak. This is when the demand is at its peak that the power system is more likely to benefit from a reduction of this power demand so that it reduces the cost of production. Furthermore, compared to \cite{MFG_revised} where the contracted target is a constant, we consider here the general case of a \textit{stochastic target process}. When $n \rightarrow \infty$, the problem can be written in terms of a Mean-Field Game model with interaction on the control.

\subsection{Extended MFG model for Demand Side Management with Cox process}
In this subsection, we first briefly describe the model in the setting of a finite population of players, and then present the MFG formulation. 

\paragraph{Model with $n$-consumers.} We assume that there are two types of consumers (\textit{active consumers} and \textit{standard consumers}). An \textit{active consumer} $i = 1, \dots, n$ enters a demand side management contract (DSM) and is characterized by two state variables $(Q^{i}, S^{\alpha^i})$. The variable $Q_{t}^{i}$ denotes the instantaneous electricity consumption of consumer $i$ at time $t$, representing the required electricity volume. \textit{Active consumers} can  deviate from their natural power demand by an amount $\alpha_{t}^{i}$, their total instantaneous consumption being $(Q_{t}^{i}+\alpha_{t}^{i})dt$. In case the instantaneous effort $\alpha_{t}^{i} > 0$ (resp. $< 0$), the consumer is anticipating (resp. postponing) specific activities which require energy, which implies that consumption is increased (resp. decreased) compared to the natural demand. The total deviation in consumption from natural power demand up to time $t$ is represented by $S_{t}^{\alpha^i}$. The second type of consumers is represented by the \textit{standard consumers}, for $i = n + 1, \dots, n+n'$, who do not optimize their consumption. They  are characterized by a single state variable $Q^{i,st}$ corresponding to the instantaneous consumption of consumer $i$ at time $t$. More precisely, the dynamics of the consumption (resp. total deviation in consumption) for consumer $i = 1,\ldots, n$ with DSM contract are given by 
\begin{align*}
dQ_t^i &= \mu(t, Q^i_t)dt + \sigma (t,Q^i_t)dW^i_t + \sigma^{0} (t,Q^i_t)dW^0_t  ,\quad Q_0^i = q_0^i ,\\
dS_t^{\alpha^i,i} &= \alpha^i_t dt, \quad S^i _0 = s_0^i,
\end{align*}
while those for any standard consumer $i=n+1,\ldots,n+n'$ are
\begin{align*}
dQ_t^{st,i} = \mu^{st}(t, Q^{st,i}_t)dt + \sigma^{st} (t,Q^{st,i}_t)dW^i_t+ \sigma^{st,0} (t,Q^{st,i}_t)dW^0_t  ,\quad Q_0^{st,i} = q_0^{st,i}.
\end{align*}
The processes $W_t^0, W_t^1,\ldots, W_t^{n+n^\prime}$ appearing above are assumed to be independent Brownian motions, and the functions $\mu,\mu^{st},\sigma,\sigma^{st},\sigma^0, \sigma^{st,0}$ are such that the above stochastic differential equations admit strong solutions.

The demand side management contract incorporates \textit{real-time pricing} and  an \textit{interruptible load} feature. First, \textit{real-time pricing} refers to the fact that consumers are charged at a  spot price $p$ which depends on the total consumption, having the role to incentivize the active consumers to reduce their consumption when it becomes too high. The associated  \textit{power cost} $c^i_t$ is a function of the total consumption of the standard consumers and those with a DSM contract. 

The interruptible load part of the contract is described as follows. At random times indicated by the Transmission System Operator (TSO) in case of supply-demand imbalance,  the total consumption of the active consumers $\sum_i (Q_t^{i}+\alpha_t^i)$ has to match a target process $\alpha_t^{tg}$, which could represent e.g. a fraction of the usual consumption. The target is maintained for a specific duration, and each agent is penalized if the total response differs from the required level of demand. %However, the energy operator can only monitor the global consumption $Q^i + \alpha^i$ for each consumer and estimate the divergence of consumption from standard levels. 
The corresponding \textit{divergence cost} $d^i$ is expressed as a function of the total consumption of the \textit{active consumers}.

The DSM contract also includes: an \textit{inconvenience cost} $g$ (associated with the efforts made by consumers to control their consumption, which increases with the instantaneous effort $\alpha^i$ and the accumulated deviations $S^{\alpha^i}$), \textit{a demand charge cost} $l$ and  a \textit{terminal cost function $h$} (which penalizes any excess or shortfall of energy consumption during the period, as it indicates that the agent did not acquire the exact amount of energy needed during the specified time frame).

\paragraph{MFG formulation.} To present the model in the MFG setting, we first introduce the probabilistic setup.

 \textit{Probabilistic setup}. Let $(\Omega, \mathcal{F}, \mathbb{P})$ be a complete probabilistic space. We assume that all stochastic processes are defined on a finite time horizon $[0, T]$ with $T > 0$.

Suppose $W^{0}$ is a Brownian motion on this space on $[0,T]$ and $\mathbb{G}^{0} \triangleq (\mathcal{G}^{0}_{t})_{t \in [0, T]}$ is the filtration generated by $W^{0}$ augmented by the $\mathbb{P}$-null sets. Let $N^{0}$ be a \textit{doubly stochastic Poisson process} (or a \textit{Cox} Process) with a $\mathbb{G}^{0}$-predictable non-negative intensity $\lambda^{0} := (\lambda_{t}^0)_{t \in [0,T]}$.  In relation to $N^0$,  we denote by $\mathbb{D}^{0} \triangleq (\mathcal{D}^{0}_{t})_{t \in [0, T]}$ the filtration generated by the \textit{Cox} Process $N^{0}$ augmented with the $\mathbb{P}$-null sets. Let $\mathbb{F}^{0} = (\mathcal{F}^{0}_{t})_{t \in [0, T]}$ denote the filtration $\mathbb{F}^{0}=\mathbb{G}^{0} \vee \mathbb{D}^{0}$, i.e. the smallest filtration containing  $\mathbb{G}^{0}$ and $\mathbb{D}^{0}$. In our setting,  $\mathbb{F}^{0}$ plays the role of the \textit{common noise filtration}.

Assume that $\mathbb{E}[\int_0^T \lambda_s^0 ds]<\infty$ for all $t \in [0,T]$, from which it follows that the \textit{compensated Poisson process}   $\tilde{N}^{0}_{t} := {N}^{0}_{t} - \int_0^t \lambda_s^{0}ds$ is a $\mathbb{F}^0$-martingale.

 We also introduce the  Brownian motions $W$ and $\bar{W}$ (representing the \textit{idiosyncratic noises} of the active and standard consumers), which are independent of $W^{0}$ and $N^{0}$. Let $\mathbb{G} \triangleq (\mathcal{G}_{t})_{t \in [0, T]}$ denote the filtration generated by $W$ and $\bar{W}$, augmented  with the $\mathbb{P}$-null sets. We denote by $\mathbb{F}^{W}$ the (completed) filtration generated by $W$. Let $(s_0, q_0)$ be two random variables independent of all the above processes. Finally, let $\mathbb{F} = (\mathcal{F}_{t})_{t \in [0, T]}$ be the smallest filtration containing  $\mathbb{F}^{0}$, $\mathbb{G}$, and the  information generated by $(s_0, q_0)$.

\textit{Representative consumer with DSM contract and representative standard consumer}. The \textit{representative consumer} involved in the DSM contract is characterized by two state variables $(Q, S^\alpha)$, with $Q_{t}$ representing the instantaneous volume of electricity needed at time $t$ and $S_{t}^\alpha$ the accumulated deviation of electricity from the natural consumption, which is controlled by a control process $(\alpha_t)$. The dynamics of the state variables of the representative consumer with \textit{DSM} contract are given by
\begin{align}\label{sys}
\begin{cases}
dQ_{t} &= \mu(t, Q_{t})dt + \sigma(t, Q_{t}) dW_{t} + \sigma^{0}(t, Q_{t})dW^{0}_{t}, \quad Q_{0} = q_{0},\\
dS^\alpha_{t} &= \alpha_{t} dt, \quad S_{0} = s_{0},
\end{cases}
\end{align}
where $(\alpha_t)$ represents the instantaneous effort.

The \textit{representative standard consumer} is characterized by only one state variable $Q^{st}$ representing their usual consumption. The dynamics of the standard consumption is then given by
\begin{align}\label{sys1}
dQ^{st}_{t} = \mu^{st}(t, Q_{t}^{st})dt + \sigma^{st}(t, Q^{st}_{t}) d\bar{W}_{t}+\sigma^{st,0}(t, Q^{st}_{t}) d{W}^0_{t} , \quad Q^{st}_{0} = q^{st}_{0},
\end{align}
where all the above coefficients are continuous in $(t,x)$ and Lipschitz continuous with respect to $x$, uniformly in $t$.

\textit{Optimization problem consumer with DSM contract and MFG equilibrium.}  As explained in the description of the $n$-player model, the demand side management model considered in this paper includes \textit{dynamic pricing} and \textit{an interruptible load} feature. To describe the interruptible load part of the contract, let $(\alpha_t^{tg})_{t \in [0,T]}$ be a given $\mathbb{G}^{0}$-adapted process.  At random times  indicated by the operator system in charge of the production-consumption balance,  the aggregated power deviation of the consumption has to match  the stochastic contracted target process $\alpha_t^{tg}$  for a predefined time period $\theta>0$. The random times correspond to the jump times of the Cox process $(N_t^0)$. In case the target is not achieved during the period $\theta$, then the representative consumer is penalized. We introduce the process $R$ which measures the time since the last DSM jump occurred. Thus, the dynamics of $R$ are given by
$$ dR_{t} = dt - R_{t^{-}} dN^{0}_{t}, \quad R_{0} = 2\theta.$$
Fix $\xi = (\xi_{t})_{t \in [0,T]}$ a $\mathbb{F}^{0}$-adapted process which represents a \textit{predetermined} power deviation and $\alpha \in \mathcal{A}$, where $\mathcal{A}$ is the set of all real-valued $\mathbb{F}^{W} \vee \mathbb{G}^{0}$-progressively measurable processes $\alpha$ such that $\mathbb{E}[\int_{0}^{T} \alpha_{t}^{2} dt] < +\infty$ and $\mathbb{E}[ |\alpha_{\tau}| \mathds{1}_{\tau < \infty}] < +\infty$ for all $\mathbb{F}^{0}$-stopping times $\tau$ with values in $[0,T] \cup \{ +\infty \}$. This set is called the set of \textit{admissible} controls. The divergence cost is then defined as follows:
$$ d^{\alpha, \xi}_{t} = ( Q_{t} + \alpha_{t} - \alpha^{tg}_{t}) f \left (\mathbb{E}[Q_{t}|\mathcal{F}_t^0] + \xi_t - \alpha^{tg}_{t} \right)J_{t}^{\theta},$$
where $ J_{t}^{\theta} = \mathds{1}_{R_{t} \leq \theta }$ (i.e. $ J_{t}^{\theta}$ is equal to one during interruptible load contract activation and zero otherwise) and $f$ is a convex growing function such as $f(0) = 0$. %$\tilde{Q}$ is the deseasonalized consumption, $\tilde{Q} := Q - \mathbb{E}[Q]$, since the power meters can only measure the global consumption  $Q + \alpha$. 

The second component of the \textit{DSM contract} is represented by the power cost $c^{\alpha,\xi}_{t}$, which defines the \textit{dynamic pricing rule} and is defined as
\begin{align*}
c^{\alpha, \xi}_{t} &= (Q_{t} + \alpha_{t}) p \left( \pi \mathbb{E}[Q_{t}^{st}|\mathcal{F}_t^0] + (1-\pi) (\mathbb{E}[Q_{t} |\mathcal{F}_t^0] + \xi_t) \right),
\end{align*}
where $p$ represents the spot price functional of the power system at which the consumers are charged, and $\pi$ represents the proportion of standard consumers with respect to DSM consumers in the total population.

Finally, we introduce the \textit{inconvenience cost} $g(\alpha_{t},S^\alpha_{t},Q_{t})$ (with $g$ convex in $\alpha_{t}$ and $S^\alpha_{t}$), the cost $l$ representing the demand charge component of the retail tariff structure, and the \textit{terminal cost} $h(S^\alpha_{T})$. 

We can now formulate the MFG problem.  For a fixed process $\xi$, the active consumer is optimizing the following functional:
\begin{equation*}
    J(\alpha;\xi) =  \mathbb{E} \left [ \int_{0}^{T} \left \{ g(\alpha_{t}, S^\alpha_{t}, Q_{t}) + l(Q_{t} + \alpha_{t}) + c^{\alpha, \xi}_t+d^{\alpha, \xi}_t \right \}dt + h(S^\alpha_{T}) \right].
\end{equation*}

Therefore, the optimization problem of the representative consumer can be then written as follows 
\begin{align}\label{opt}
V^{MFG}(\xi) = \inf_{\alpha \in \mathcal{A}} J^{MFG}(\alpha; \xi)
\end{align}
\begin{definition}[Mean-field Nash Equilibrium]
The solution $\alpha^{\star}$ to problem $\eqref{opt}$ is called a mean field \textit{Nash equilibrium} if $\mathbb{E}[\alpha^{\star}_t|\mathcal{F}_t^0]=\xi_t$ a.s. for all $0 \leq t \leq T$. 
\end{definition}
\begin{remark}
    Notice that, in contrast to \cite{MFG_revised} where the target $\alpha^{tg}$ is only considered to be a constant, we consider here a $\mathbb{G}^{0}$-adapted target process $\alpha^{tg} = (\alpha^{tg}_{t})_{t \in [0,T]}$. We also assume that the jump times correspond to the ones of a Cox process (time non-homogeneous Poisson process), compared to \cite{MFG_revised} where it is supposed that they come from a Poisson process.
\end{remark}

\subsection{Characterization of the MFG equilibrium with Cox Process} \label{CharCox}
In this Section, we first provide a characterization of the MFG equilibrium in a general setting, and then focus on the linear-quadratic model. We  introduce the following sets, which will be used throughout the rest of the paper:
\begin{itemize}
    \item $\mathcal{S}^2$ is the set of $\mathbb{F}$-adapted càdlàg real-valued processes $Y$ such that $\mathbb{E}\left [\sup_{0 \leq t \leq T} |Y_t|^2 \right] < + \infty$.
    \item $\mathcal{H}^2$ is the set of $\mathbb{F}$-predictable real-valued  processes $q$ such that $\| q \|^{2} : = \mathbb{E}[\int_{0}^{T} |q_{t}|^{2} dt] < + \infty$.
    \item $\mathcal{H}_{\lambda^{0}}^2$ is the set of $\mathbb{F}$-predictable real-valued processes $\nu^{0}$ such that $\| \nu^{0} \|^{2}_{\lambda^{0}} : = \mathbb{E}[\int_{0}^{T} |\nu^{0}_{t}|^{2} \lambda^{0}_{t} dt] < + \infty$.
    \item $L^2(\mathcal{F}_{T})$ is the set of $\mathcal{F}_T$-measurable real-valued random variables $\xi$ such that $\mathbb{E}[\lvert \xi \vert^2] < +\infty$.
\end{itemize}

    In the sequel, given a $\mathcal{B}([0, T]) \otimes \mathcal{F}$-measurable process $X$ such that $E[|X_\tau| 1_{\tau < \infty}] < \infty$ for all $\mathbb{F}^0$-stopping times $\tau$ with values in $[0, T] \cup \{+\infty\}$, we will denote by $\widehat{X}$ the  optional projection of $X$ with respect to the filtration $\mathbb{F}^0$, i.e. $\widehat{X}$ is the unique (up to indistinguishability) $\mathbb{F}^0$-optional process such that $\widehat{X}_\tau 1_{\tau < \infty} = \mathbb{E} \left [X_\tau 1_{\tau < \infty}|\mathcal{F}^0_{\tau} \right]$ a.s. for all $\mathbb{F}^0$-stopping times $\tau$ with values in $[0, T] \cup \{+\infty\}$ (cf. Section 2 in \cite{bremaud_yor_1978}). 

\begin{assumption} \label{assum:MFG} We make the following assumptions:
\begin{itemize}
    \item $g$, $l$ and $h$ have at most quadratic growth and are strictly convex.
    \item $p$ and $f$ have at most linear growth.
    \item $g$, $p$, $f$, $l$ and $h$ are differentiable.
\end{itemize}
\end{assumption}
We now give the following characterization of a MFG equilibrium.
\begin{theorem}[Characterization of the mean-field game equilibrium]\label{existunique}
Let $\hat{\xi}$ be a given $\mathbb{F}^{0}$-adapted real valued process and $x_0 = (s_0, q_0, q^{st}_0)$ be a random vector independent of $\mathbb{F}^{0}$. Assume that $\alpha \rightarrow J^{MFG}(\alpha; \hat{\xi}) $ is strictly convex. If there exists a control $\alpha^{\star} \in \mathcal{A}$ which minimizes $\alpha \rightarrow J^{MFG}(\alpha; \hat{\xi}) $ and if $(S^{\alpha^{\star}}, Q, Q^{st})$ is the state process associated to the initial condition $x_0$, optimal control $\alpha^{\star}$ and the dynamics \eqref{sys}-\eqref{sys1}, then there exists a unique solution $(Y^{\star}, q^{0, \star}, q^{\star}, \nu^{0, \star}) \in \mathcal{S}^2 \times (\mathcal{H}^2)^2 \times \mathcal{H}_{\lambda^{0}}^2$ of the following BSDE with jumps:
\begin{align} \label{eqn:sol}
\begin{cases}
    -dY_{t}^{\star} = \partial_{\alpha}g(\alpha_t^{\star}, S_{t}^{\alpha^{\star}}, Q_{t})dt - q^{0, \star}_{t} dW_{t}^{0} - q^{st, \star}_{t} d \bar{W}_{t} - \nu^{0, \star}_{t} d\tilde{N}^{0}_{t}, \\
    Y_{T}^{\star}  = \partial_{x}h(S_{T}^{\alpha^{\star}}),\\
\end{cases}
\end{align}
satisfying the coupling condition
\begin{equation} \label{eqn:coupling}
    \partial_{\alpha}g(\alpha_t^{\star}, S_{t}^{\alpha^{\star}}, Q_{t}) + \partial_{\alpha} l(Q_{t} + \alpha^{\star}_{t}) + p(\pi \widehat{Q}^{st}_t +(1- \pi)(\widehat{Q}_{t} + \hat{\xi}_{t})) + Y_{t}^{\star} + J^{\theta}_{t}f(\widehat{Q}_{t} + \hat{\xi}_{t} - \alpha^{tg}_{t}) = 0
\end{equation}
\noindent Conversely, assume that there exists  $(\alpha^{\star}, S^{\alpha^{\star}}, Y^{\star}, q^{0, \star}, q^{\star}, \nu^{0, \star}) \in \mathcal{A} \times (\mathcal{S}^2)^2 \times (\mathcal{H}^2)^2 \times \mathcal{H}_{\lambda^{0}}^2$ satisfying the FBSDE $\eqref{eqn:sol}$ and the coupling condition $\eqref{eqn:coupling}$, then $\alpha^{\star}$ is the optimal control minimizing $\alpha \rightarrow J^{MFG}(\alpha; \hat{\xi}) $ and $S^{\alpha^{\star}}$ is the optimal trajectory.

If additionally $\widehat{\alpha}^{\star}_{t} = \hat{\xi}_{t}$ a.s for all $t \in [0, T]$, then $\alpha^{\star}$ is a Mean-field equilibrium.
\end{theorem}

\begin{proof}
    Under  Assumption \ref{assum:MFG} and using similar arguments as in  \cite{jumpdef} to prove existence and uniqueness results for BSDEs driven by Cox processes, we conclude that the BSDE defined in the theorem is well-posed. Using this result, the proof follows the same steps as in Theorem 3.1 in \cite{MFG_revised} , and we therefore omit it.
\end{proof}
\paragraph{Semi-explicit representation of the MFG equilibrium in the linear quadratic case}
We shall provide here a semi-explicit characterization of the equilibrium in the linear-quadratic setting, which is ensured by the following assumption.

\begin{assumption}
Let  $(\chi_{t})_{t \in [0, T]}$ and $(\chi^{st}_{t})_{t \in [0, T]}$ be two continuous deterministic processes. We suppose that the following assumptions are satisfied: \\
1. $\mu(t,q) = \mu (\chi_{t} - q) $, $\mu^{st}(t,q) =  \mu^{st} (\chi^{st}_{t} - q)$, $\sigma(t,q) = \sigma $, $\sigma^{st}(t,q) = \sigma^{st}  $, $\sigma^{st, 0}(t,q) = \sigma^{st, 0}$, and $\sigma^{0}(t,q) = \sigma^{0}$, with $\mu, \mu^{st}, \sigma,\sigma^{0},\sigma^{st}>0$.\\
2. $g(a,s,q) = \frac{A}{2}a^{2} + \frac{C}{2} s^{2}$ with $A,C > 0$. \\
3. $l(x) = \frac{K}{2} x^{2}$ with $K \geq 0$.\\
4. $f(a) = f_{0} + f_{1}a$ with $f_{i} \in \mathbb{R}, i = 0,1$ and $f_{1} \geq 0$. \\
5. $p(q) = p_{0} + p_{1}q$ with $p_{0} \in \mathbb{R}$, and $p_{1} > 0$. \\
6. $h(s) = h_{0} + h_{1}s + \frac{h_{2}}{2}s^{2}$, with $h_{i} \in \mathbb{R}, i = 0,1,2$ and $h_{2} \geq 0$.
\end{assumption}

Following the approach used in \cite{MFG_revised} in the particular case of a Poisson process, in the linear-quadratic setting we look for solutions taking the  form:
$$ \widehat{Y}_{t} = \bar{\phi}_{t} S^{\widehat{\alpha}}_{t} + \bar{\psi}_{t} \text{ and } Y_{t} = \phi_{t} S^{\alpha}_{t} + \psi_{t},$$ 
with $(\bar{\phi}_{t}, 0,  \widehat{\xi}^0_t,\widehat{\xi}^{0,N}_t)$, $(\bar{\psi}_{t},\widehat{\eta}^0_t,0, \widehat{\eta}^{0,N}_t)$, $(\phi_{t},0,0,0)$ and $(\psi_{t}, \eta_t^{0}, \eta_t, \eta_t^{0,N})$ the unique solutions in $\mathcal{S}^2 \times (\mathcal{H}^2)^2 \times \mathcal{H}^2_{\lambda^0}$  of the following BSDEs driven by a Cox process:
\begin{align*}
    d \phi_{t} &= \left (-C + \frac{1}{A+K} \phi_{t}^{2} \right)dt, \quad \phi_{T} = h_{2},\\
    d \psi_{t} &= \frac{\phi_{t}}{A+K} \left [ K Q_{t} + p_{0} + \pi p_{1} \widehat{Q}^{st}_{t} + ((1 - \pi)p_{1} + K) (\widehat{Q}_{t} + \widehat{\alpha}_{t}) + \right.\\
    &\left. J_{t}^{\theta}(f_{0} + f_{1} (\widehat{Q}_{t} + \widehat{\alpha}_{t} - \alpha^{tg}_{t})) + \psi_{t}\right]dt
    + \eta^{0}_{t} dW_{t}^{0} + \eta_{t} dW_{t} + \eta^{0,N}_{t} d\tilde{N}_{t}^{0}, \quad \psi_{T} = h_{1},\\
    d \bar{\phi}_{t} &= \left (-C + \frac{1}{K^{\theta}_{t}} \bar{\phi}_{t}^{2} \right) + \widehat{\xi}_{t}^{0} dW_{t}^{0} + \widehat{\xi}_{t}^{0,N} d \tilde{N}_{t}^{0}, \quad \bar{\phi}_{T} = h_{2},\\
    d \bar{\psi}_{t} &= \frac{\bar{\phi}_{t}}{K^{\theta}_{t}} \left [ p_{0} + \pi p_{1} \widehat{Q}^{st}_{t} + ((1 - \pi)p_{1} + K) \widehat{Q}+ J_{t}^{\theta}(f_{0} + f_{1} (\widehat{Q}_{t} - \alpha^{tg}_{t})) + \bar{\psi}_{t}\right]dt \\
    &+ \widehat{\eta}^{0}_{t} dW_{t}^{0} + \widehat{\eta}^{0,N}_{t} d\tilde{N}_{t}^{0}, \quad \bar{\psi}_{T} = h_{1},
\end{align*}
where $K^{\theta}_{t} = A + K + (1-\pi) p_{1} + f_{1} J_{t}^{\theta} $.

The wellposedness of the above BSDEs follows by an adaptation of the theorems provided in \cite{jumpdef}. 
By using the ansatz and replacing it in the projected coupling condition, we obtain:
\begin{align}\label{projequil}
\begin{small}
\widehat{\alpha}_{t} = - \frac{1}{K^{\theta}_{t}} \left ( p_{0} + \pi p_{1} \widehat{Q}^{st} + ((1-\pi)p_{1} + K) \widehat{Q}_{t} + \widehat{Y}_{t} + (f_{0} + f_{1} (\widehat{Q}_{t} - \alpha^{tg}_{t}))J_{t}^{\theta} \right).
\end{small}
\end{align}
Finally, by using the expression of $\widehat{\alpha}$ and again the ansatz and the coupling condition, we finally obtain that the \textit{MFG equilibrium} $\alpha$ admits the following representation:

\begin{align}\label{equil}
\begin{split}
\alpha_{t} &= - \frac{1}{A + K} \left ( K Q_{t} + p_{0} + \pi p_{1} \widehat{Q}^{st} + ((1-\pi)p_{1} + K) (\widehat{Q}_{t} + \widehat{\alpha}_{t}) + Y_t + (f_{0} + f_{1} (\widehat{Q}_{t} + \widehat{\alpha}_{t} \right.\\ &\left. - \alpha^{tg}_{t}))J_{t}^{\theta} \right).
\end{split}
\end{align}

\subsection{Aggregator problem and Price of Anarchy}
In this part, we consider the point of view of an aggregator and characterize his optimal strategy, as well as discuss the related price of anarchy.
\paragraph{Aggregator problem}
We now introduce the following mean-field control problem of an aggregator who plays the role of a central planner who coordinates all the DSM consumers in the system, without taking into account the non-active consumers. The associated value function of the aggregator is given by
\begin{align*}
    V^{{MFC}^{agg}} =  \inf_{\alpha \in \mathcal{A}}\mathbb{E} & \left [ \int_{0}^{T} \left \{ g(\alpha_{t}, S^\alpha_{t}, Q_{t}) + l(Q_{t} + \alpha_{t}) + c^{\alpha,\widehat{\alpha}}_{t} + d^{\alpha,\widehat{\alpha}}_{t} \right \}dt + h(S^\alpha_{T}) \right].
\end{align*}

\noindent The solution to this optimization problem is called the $MFC^{agg}$ optimal control. Using a similar proof to the one of Theorem \ref{existunique}, we have the following characterization of the optimal control.

\begin{theorem}[Characterization of the aggregator's mean field control] \label{CHMFG}
Let $x_0=(s_0, q_0, q_0 ^{st})$ be a random vector independent of $\mathbb{F}^0$. \textcolor{black}{Assume that the map $\alpha \mapsto J^{MFC}(\alpha)$ is strictly convex}. If there exists a control $\alpha^\star \in \mathcal{A}$ which minimizes the map $\alpha \mapsto J^{MFC}(\alpha)$ and if $(S^{\alpha^\star}, Q, Q^{st})$ is the state process associated to the initial condition $x_0$, control $\alpha^\star$ and the dynamics \eqref{sys}-\eqref{sys1}, then there exists a unique solution $(Y^\star,q^{0, \star},q^{\star}, \nu^{0,\star}) \in \mathcal{S}^2 \times (\mathcal{H}^2)^2 \times \mathcal{H}_{\lambda^0}^2$ of the BSDE with jumps
\begin{align}\label{MFCagg}
\begin{cases}
    -dY^\star_t=\partial_x g(\alpha_t^{\star}, S_t ^{\alpha^\star} , Q_t)dt-q_t^{0, \star}dW_t^0-q^\star_t dW_t - \nu^{0,\star}_t d\widetilde N^0_t,\\ 
    Y^\star_T=\partial_x h(S_T^{\alpha^\star}), 
\end{cases}
\end{align}
satisfying the coupling condition
\begin{align}\label{coupling2}
\partial_\alpha g(\alpha^{\star}_t, S^{\alpha^{\star}}_t, Q_t)+\partial_\alpha l(Q_t+ \alpha^{\star}_t)+p\left(\pi\widehat{Q}^{st}_t +(1-\pi)(\widehat Q_t + \widehat{\alpha}^\star_t)\right) \nonumber \\
+ (\widehat Q_t + \widehat \alpha^{\star}_t)\partial_\alpha p(\pi \widehat Q_t^{st} + (1-\pi)(\widehat Q_t + \widehat \alpha^{\star}_t))  \nonumber\\
 \quad \quad +Y_t^{\star} + J_t^{\theta}f(\widehat{Q_t} +  {\widehat \alpha^\star_t} - {\color{black}\alpha^{tg}})+J_t^{\theta}(\widehat{Q_t} +{\widehat \alpha^\star_t} - {\color{black}\alpha^{tg}}) \partial_\alpha f(\widehat{Q_t} +  \textcolor{black}{\widehat \alpha^\star_t} - {\color{black}\alpha_t^{tg}}) =0,
\end{align}

{\color{black}with $\widehat{\alpha}^\star$ the optional projection of $\alpha^\star$ with respect to $\mathbb F^{0}$.}
Conversely, assume that there exists $\left(\alpha^\star,S^{\alpha^\star}, Y^\star, q^{0,\star},q^{\star}, \nu^{0,\star} \right) \in \mathcal A \times (\mathcal{S}^2)^2 \times (\mathcal{H}^2)^2\times \mathcal{H}^2_{\lambda^0}$ satisfying the coupling condition $\eqref{coupling2}$, as well as the FBSDE \eqref{sys}-\eqref{MFCagg}, then $\alpha^\star$ is the optimal control minimizing the map $\alpha \mapsto J^{MFC}(\alpha)$ and $S^{\alpha^\star}$ is the optimal trajectory.
\end{theorem}
\begin{remark}\label{MFG/MFC} As observed in \cite{MFG_revised}, by comparing the coupling conditions \eqref{eqn:coupling} and \eqref{coupling2}, the optimal control for the $MFC^{agg}$ problem in the linear quadratic setting with pricing rules $p^{MFC^{agg}}(\pi \widehat{Q}^{st} + (1-\pi)Q) = p_0 + p_1(\pi \widehat{Q}^{st} + (1-\pi)Q)$ and $f^{MFC^{agg}}(Q) = f_0 + f_1Q$ corresponds to the $MFG$ equilibrium for the problem with pricing rules $p^{MFG}(\pi \widehat{Q}^{st} + (1-\pi)Q) = p_0 + 2p_1(1 - \pi)Q + p_1\pi \widehat{Q}^{st}$ and  $f^{MFG}(Q) = f_0 + 2f_1 Q $. 
\end{remark}

\paragraph{Price of Anarchy} %The price of anarchy is a measure of the inefficiency of a game when each player acts selfishly without considering the impact of their actions on the system as a whole. In the context of mean-field control, it is the ratio of the value of a Nash equilibrium control policy to that of the socially optimal policy.

%From the perspective of an aggregator, the price of anarchy measures the cost of decentralized decision-making by individual players versus that of a centralized control policy that accounts for player interactions. It captures the loss in performance due to the lack of coordination among players.

The price of anarchy is defined as the ratio of a worst case social cost computed for a mean field game equilibrium to the optimal social cost as computed by a central planner.

For our problem, the expression for the price of anarchy takes the following form:

$$ \text{PoA} = \frac{V^{MFG}(\widehat{\alpha}^{\star})}{V^{MFC^{agg}}}, $$
where $\alpha^{\star}$ is the MFG Nash equilibrium.

\subsection{Deep learning algorithms for the MFG and MFC problem}
In this section, we design several numerical algorithms to compute in the linear-quadratic setting the MFG equilibrium and the mean-field optimal control for the aggregator's problem. The algorithms are based on the machine learning solvers introduced in the first part of the paper and extended to the case of a time-inhomogeneous Poisson process with \textit{stochastic intensity} of jumps.

\paragraph{Characterization of the MFG equilibria via a multi-dimensional coupled FBSDE with jumps} Using the results from the previous section, the MFG equilibria in the linear-quadratic case can be expressed through the following multi-dimensional fully-coupled FBSDE system:
\begin{align}\label{MFGsys1}
\textbf{(MFG)}\begin{cases}
dQ_{t} = \mu(\chi_{t} - Q_{t})dt + \sigma dW_{t} + \sigma^{0}dW^{0}_{t}, \\
d\widehat{Q}_{t} = \mu(\chi_{t} - \widehat{Q}_{t})dt  + \sigma^{0}dW^{0}_{t}, \\
dQ^{st}_{t} = \mu^{st}(\chi^{st}_{t} - Q^{st}_{t})dt + \sigma^{st} d\bar{W}_{t} + \sigma^{st, 0}dW^{0}_{t}, \\
d\widehat{Q}^{st}_{t} = \mu^{st}(\chi^{st}_{t} - \widehat{Q}^{st}_{t})dt  + \sigma^{st, 0}dW^{0}_{t}, \\
dR_{t} = dt - R_{t^{-}} dN^{0}_{t},\\
dS^{\alpha^{\star}}_{t} = - \frac{1}{A + K} \left ( K Q_{t} + p_{0} + \pi p_{1} \widehat{Q}^{st} + ((1-\pi)p_{1} + K) (\widehat{Q}_{t} + P(t,\widehat{Q}_t,\widehat{Q}^{st}_t, \widehat{Y}_t,R_t)) + Y_t + (f_{0} + f_{1} (\widehat{Q}_{t} \right.\\ \left. + P(t,\widehat{Q}_t,\widehat{Q}^{st}_t, \widehat{Y}_t,R_t)  - \alpha^{tg}_{t})) \mathds{1}_{R_{t} \leq \theta } \right) dt,\\
dS^{\widehat{\alpha}^{\star}}_{t} = P(t,\widehat{Q}_t,\widehat{Q}^{st}_t, \widehat{Y}_t,R_t)dt, \\
-dY_{t} = C S_{t}^{\alpha^{\star}} dt - q_{t}^{0}dW_{t}^{0} - q_{t}dW_{t} - \nu^{0}_{t}d\tilde{N}_{t}^{0},\\
-d\widehat{Y}_{t} = C S_{t}^{\widehat{\alpha}^{\star}} dt - \widehat{q}_{t}^{0}dW_{t}^{0} - \widehat{\nu}^{0}_{t}d\tilde{N}_{t}^{0}, \\
Q_{0} = q_{0}, \quad Q^{st}_0 = q^{st}_0, \quad R_{0} = 2\theta, \quad S^{\alpha^{\star}}_{0} = s_{0}, \quad Y_{T} = h_{1} + h_{2} S_{T}^{\alpha^{\star}}, \\ \widehat{Q}_{0} = q_{0}, \quad \widehat{Q}^{st}_0 = q^{st}_0,  \quad S^{\widehat{\alpha}^{*}}_{0} = s_{0}, \quad \widehat{Y}_{T} = h_{1} + h_{2} S_{T}^{\widehat{\alpha}^{\star}},
    \end{cases}
\end{align}
where
\begin{align*}
    P(t,\widehat{Q}_t,\widehat{Q}^{st}_t, \widehat{Y}_t,R_t)&:= - \frac{1}{A + K + (1-\pi) p_{1} + f_{1} \mathds{1}_{R_{t} \leq \theta } } \left ( p_{0} + \pi p_{1} \widehat{Q}^{st} + ((1-\pi)p_{1} + K) \widehat{Q}_{t} + \widehat{Y}_{t}  \right. \nonumber \\ & \left. + (f_{0} + f_{1} (\widehat{Q}_{t} - \alpha^{tg}_{t}))\mathds{1}_{R_{t} \leq \theta } \right).
\end{align*}

We are thus led to solve a fully-coupled multi-dimensional FBSDE driven by a doubly stochastic Poisson process (which admits an unique solution, by using the results from the part on the \textit{Characterization of the MFG equilibrium in the linear quadratic case} from subsection \ref{CharCox}). For a generic fully-coupled system of FBSDEs driven by a doubly Poisson process, the discretized version takes the form
\begin{align}
\label{eqn:discreteMFG1}
\begin{cases}
      X_{i+1}^{\pi} =  X_{i} + b(t_i,X_i^{\pi},Y_{i}) \Delta t_{i}  + \sigma (t_i,X_i^{\pi})   \Delta W_i +  \sigma^0 (t_i,X_i^{\pi})   \Delta W^0_i + \beta(t_i,X_i^{\pi}) dN^0_i  , \\
      Y_{i+1}^{\pi} \approx  Y_{i}^{\pi} - f(t_i,X_{i}^{\pi} ,Y_{i}^{\pi})\Delta t_{i}  + Z_i^{\pi} \Delta W_i + Z_i^{0,\pi} \Delta W^0_i
      + U_i^{\pi} dN^0_i - U_i^{\pi} \lambda^0_i \Delta t_i, \\
      X_0^{\pi} = \xi, \quad Y_M^{\pi} = g(X_M^{\pi}),\\
      i = 0, \cdots, M-1.
\end{cases}
\end{align}

%Cox Process, where $\mathcal{J}$ represents the associated random Poisson measure with compensator $\lambda^0_t \delta_{1}(de)dt$, and $\delta_{1}$ denotes the Dirac measure on $\mathbb{R}^{*}$ centered on 1. Therefore, it is preferable to use the \textbf{discretization \eqref{eqn:discrete1}}, which, with this random measure, can be expressed as follows:

Notice that in this setting, we do not require specific methods to estimate the compensator as it is directly given in this model. Thus, there is no need to test the two variants for the \textit{Sumlocal} and \textit{SumMultiStep} methods when comparing the deep learning solvers.

Furthermore, we make the following assumption on the intensity $\lambda^0$:
\begin{assumption}
There exists a continuous function $\bar{\lambda}:\mathbb{R} \mapsto \mathbb{R}$ such that
\begin{align}
\lambda_t^0=\bar{\lambda}(\widehat{Q}_t).
\end{align}
\end{assumption}
Finally, in view of Remark \ref{MFG/MFC}, the computation of the mean-field optimal control of the aggregator can be done through the same  multi-dimensional fully-coupled FBSDE system \eqref{MFGsys1}, but with different coefficients.

%Furthermore, the existence and uniqueness of the fully-coupled multi-dimensional FBSDE $\eqref{MFGsys1}$ are deduced from the following observations: The results given in the \textbf{Characterization of the MFG equilibrium in the linear quadratic case} part in the subsection \ref{CharCox} guarantees the existence and uniqueness of the backward components, especially, since the linear SDEs of $S^{\widehat{\alpha}^{\star}}$ and $S^{\alpha^{\star}}$ can be solved explicitely using the same method as in \cite{MFG_revised} (subsection 4.2). However, the other forward components are decoupled, hence the existence and uniqueness are trivial in the linear quadratic setting. 

\paragraph{Numerical results and comparison between the deep-learning solvers.} We shall now perform a detailed analysis of the convergence and stability of the five different algorithms to solve the FBSDE system associated to the computation of the MFG equilibria. To do so, we first set the hyper-parameters.
\begin{center}
\begin{tabular}{| c | c |}
\hline
Parameter & value\\
\hline
 $m$ & 20\\
 $L$ & 2\\
 NbTraining & 10000\\
\hline
\end{tabular}
\quad
\begin{tabular}{| c | c |}
\hline
Parameter & value\\
\hline
 B & 64\\
 lRate & 0.01/0.007\\
 $\sigma_{a}$ & tanh\\
\hline
\end{tabular}
\quad
\begin{tabular}{| c | c |}
\hline
\end{tabular}
\end{center}
% \begin{center}
% \begin{tabular}{| c | c |}
% \hline
% Parameter & value\\
% \hline
%  NbNeuronsHat & 20\\
%  NbNeurons & 20\\
%  NbLayersHat & 2\\
%  NbLayers & 2\\
%  NbTraining & 10000\\

% \hline
% \end{tabular}
% \quad
% \begin{tabular}{| c | c |}
% \hline
% Parameter & value\\
% \hline
%  BatchSize & 64\\
%  rafCoef & 1\\
%  lRate & 0.01/0.007\\
%  activationHat & tanh\\
%  activation & tanh\\
% \hline
% \end{tabular}
% \quad
% \begin{tabular}{| c | c |}
% \hline
% \end{tabular}
% \end{center}
The compensator having an analytical form, we have adapted the algorithms and have taken batches of small sizes. 
In contrast to \cite{MFG_revised}, we consider a random intensity of jumps $(\lambda_t^0)$ and a  target process $(\alpha_t^{tg})$ which are given by
\begin{align*}
    \lambda^{0}_{t} = e^{-\frac{\gamma}{2}} (e^{\gamma \widehat{Q}_{t}} - 1),\,\,\,
    \alpha^{tg}_{t} =  \beta \mathbb{E} \left[\widehat{Q}_{t} \right].
\end{align*}

The other parameters associated to the model are inspired from \cite{MFG_revised} and are given below: 
\begin{center}
\begin{tabular}{| c | c |}
\hline
Parameter & value\\
\hline
 $T$ & 2 days\\
 $n_{steps}$ & 96 half-hours\\
 $A$ & 150\\
 $C$ & 80\\
 $K$ & 50\\
 $\chi^{st}$ & $\chi$\\
$\sigma^{0} = \sigma^{st,0}$  & 0.1\\
\hline
\end{tabular}
\quad
\begin{tabular}{| c | c |}
\hline
Parameter & value\\
\hline
$ \sigma$ & 0.3\\
$\sigma^{st}$ & 0\\ 
 $\mu$ & 5\\
 $h_{0} = h_{1}$ & 0\\
 $h_{2}$ & 600\\
 $\theta$ & 0.12 hours\\
 $s_{0}$ & 0\\
\hline
\end{tabular}
\quad
\begin{tabular}{| c | c |}
\hline
Parameter & value\\
\hline
 $ f_{0} $ & 0\\
 $f_{1}$ & 10000\\
 $p_{0}$ & 6.16 \euro/MWh \\
 $p_{1}$ & 87.43 \euro/MWh$^2$ \\
 $q_0 = q^{st}_0$ & $\chi_0$\\
 $\gamma$ & 30 \\
 $\beta$ & 0.8\\
\hline
\end{tabular}
\end{center}

The function $(\chi_{t})$ corresponds to the consumption seasonality observed from the data from \cite{MFG_revised}. 

We shall now compare the initial values of the backward components through each epoch. Notice that, between two epochs, $100$ stochastic gradient descents are performed. 
\begin{figure}[H]
    \centering
    \includegraphics[width= 16 cm]{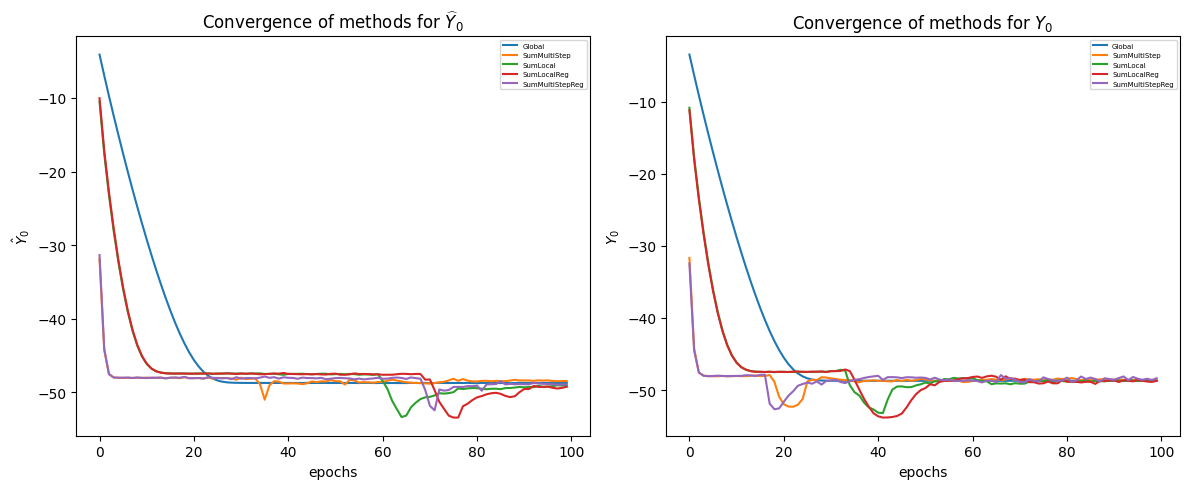}
    
    \caption{Convergence of the 5 algorithms in the MFG model}
    \label{fig:cvgMFG}
\end{figure}
According to various benchmarks and the convergence results in Figure \ref{fig:cvgMFG}, it is obvious  that the comparison results are similar to the pricing models.  The \textit{Global} method is the most stable one and provides a good approximation with a large learning rate which makes up for the problem of initializing low values. \textit{MultiStep} method and its regression version \textit{MultiStepReg} converge after few epochs with a learning rate considerably lower than the one used in the Global method. They present a good trade-off between convergence speed and stability. Finally, \textit{SumLocal} method and its regression version performed poorly in terms of stability as is obvious from the figures above.

\begin{remark}
 In the particular case of a Poisson process with a constant intensity and a constant consumption rate target $\alpha^{tg}$, we have compared our results with the ones obtained in \cite{MFG_revised} (which were obtained by combining a tree approximation of the martingales, as in e.g. \cite{dumitrescu2016numerical, dumitrescu2016reflected}, and the Monte Carlo method). We observed that our results were coherent with the ones provided in  \cite{MFG_revised}.
\end{remark}
%This represents the fact that the system aims at achieving a reduction of 50 percent the power demand when the global demand is high. Let us present some numerical results achieved using the \textbf{Global algorithm}. Let's note that the results obtained by the other methods are quite similar. Indeed, we performed a comparison and all the methods converge. We found out that the Global algorithm gives excellent results as presented in the \textbf{Convergence and stability} subsection below.

\subsection{Interpretation of numerical results from a modelling perspective} In this Section, we provide an economic interpretation of our numerical results, which are computed using the \textit{Global method}. The results are illustrated on two typical customers whose power consumption are represented in Figure \ref{fig:subfigures}. Consumer 2 shows a typical power consumption profile with two peaks of consumption in the morning and in the evening. Consumer 2  needs more electricity during the first day than the consumers' average consumption, whereas Consumer 1 consumes very little during the first morning. As expected, the intensity of jumps is very high when the consumption is at its highest level. 
\begin{figure}[h!] %[htbp!]
    \centering
    \begin{subfigure}[b]{0.5\textwidth}
        \centering
        \includegraphics[width=\linewidth]{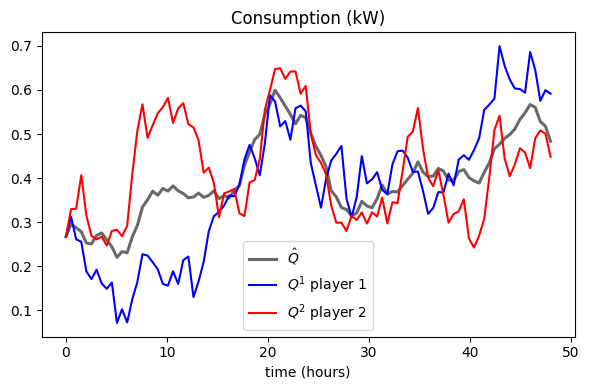}
    \end{subfigure}%
    \begin{subfigure}[b]{0.5\textwidth}
        \centering
        \includegraphics[width=\linewidth]{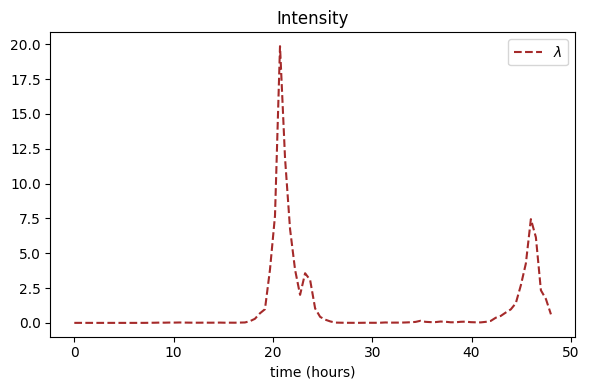}
    \end{subfigure}
      \caption{Trajectories over 48 hours of the consumption for 2 different consumers in kW (upper figure left) and the common intensity of jumps for the divergence costs (upper figure right). }
    \label{fig:subfigures}
\end{figure}    
 %   \vspace{1ex} 
 \begin{figure}[h!]   
        \centering
        \includegraphics[width=0.5\linewidth]{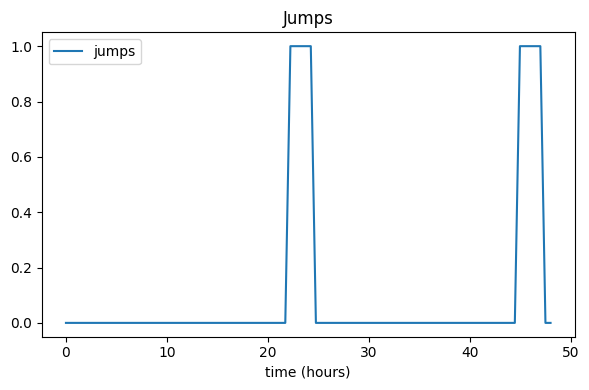}
    
    \caption{One trajectory of DSM activation jumps issued from the  intensity presented in the previous figure.}
    \label{fig:jump_trajectory}
\end{figure}
\\
The following results present how these two typical consumers optimize their consumption when two activations of the DSM contract happen following the DSM activation scenario presented in Figure \ref{fig:jump_trajectory} .

The illustrations show that the consumers react as expected: when they are exposed to dynamic pricing only (no activation of jump DSM), they smooth their consumption over the period as illustrated in Figure \ref{fig:optimise_traj_dynamic_jump_separated} \textcolor{blue}{(b)}. When they are exposed to divergence cost only, their average consumption perfectly matches the random target $\alpha^{tg}$, whereas the individual consumption of the consumers can differ from the target.

\begin{figure}[H]
    \centering
    \begin{subfigure}[b]{0.5\textwidth}
        \centering
        \subfloat[With divergence cost and no dynamic pricing]{{\includegraphics[width=\linewidth]{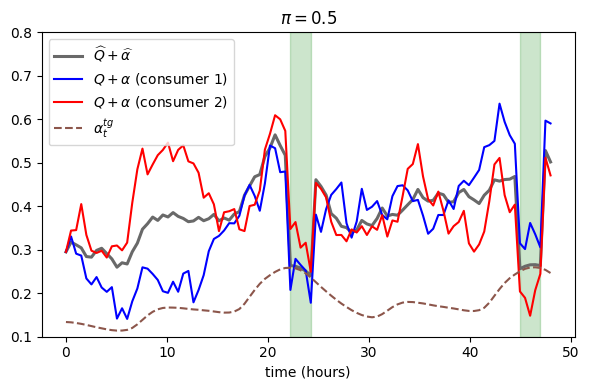}}}
    \end{subfigure}%
    \begin{subfigure}[b]{0.5\textwidth}
        \centering
        \subfloat[Dynamic pricing and without divergence cost]{{\includegraphics[width=\linewidth]{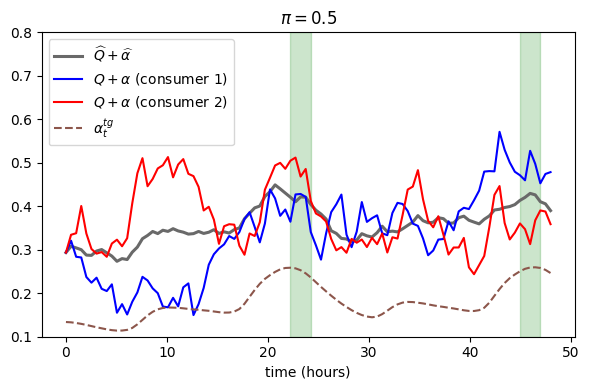}}}
    \end{subfigure}\\[1ex]
    \caption{Trajectories of $\widehat{Q} + \widehat{\alpha}$ and $Q + \alpha$ (in kW) for two consumers in the MFG setting when these consumers have no dynamic pricing but only the control with respect to the divergence cost (a) and when have dynamic pricing only (b). DSM activations are represented by the green bar.}
    \label{fig:optimise_traj_dynamic_jump_separated}
\end{figure}

When consumers are exposed to both dynamic pricing and divergence cost activation, they combine the two behaviours observed above. Their resulting  consumption is presented in Figure \ref{fig:optimise_traj_dynamic_jump_together}.
\begin{figure}[H]%[htbp]
    \centering
    \subfloat[With divergence cost and dynamic pricing]{{\includegraphics[width=8cm]{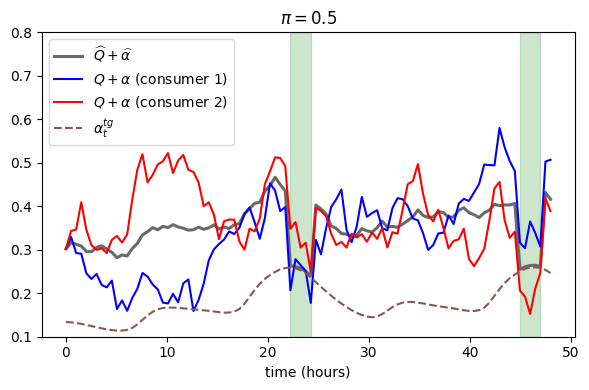}}}
    \caption{Trajectories of $\widehat{Q} + \widehat{\alpha}$ and $Q + \alpha$ (in kW) for two consumers in the MFG setting.}
    \label{fig:optimise_traj_dynamic_jump_together}
\end{figure}

We then analyze how the spot price reacts to the DSM contract (see Figure \ref{fig:spot_price}). We can observe that the proportion of consumers with DSM contract (the lower $\pi$, the more widespread the DSM contract within the global population) directly impacts how much spot price is smoothed and how much peak prices are reduced.\\

\begin{figure}[H]
    \centering
    {{\includegraphics[width=8cm]{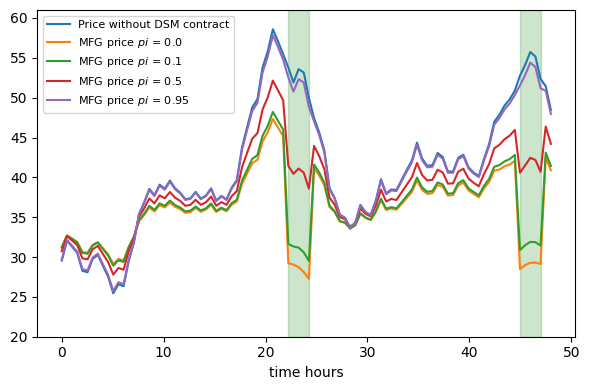} }}
    \caption{Trajectories of the price $p$ for four different proportions of active consumers in the MFG setting.}
    \label{fig:spot_price}
\end{figure}

\textit{A comparison between MFCagg and MFG.} We now provide a comparative analysis between the levels of the consumption and corresponding prices in the case when the optimization problem is either implemented from an aggregator perspective, i.e. a MFC problem, or is solved in the MFG setting. We can observe that 
when consumers are not selfishly optimizing their power consumption, but are guided by an aggregator they make greater effort to reduce their consumption (see figure \ref{fig:MFG_MFCagg}). This efficiency can be attributed to better coordination among consumers in the MFC problem. Naturally, as the prices follow the same trend as the power consumption, we observe that the prices are cheaper when there is an aggregator. \\
\begin{figure}[H]
    \centering
    {{\includegraphics[width=16cm]{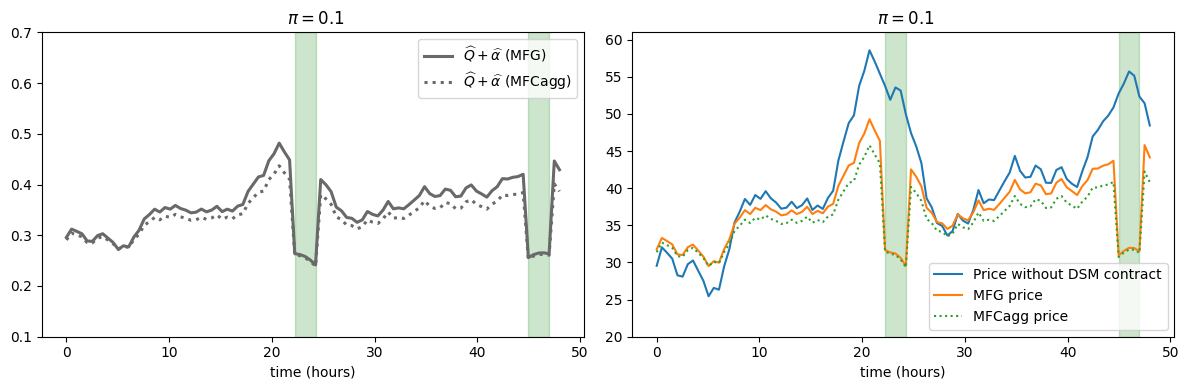} }}
    \caption{Trajectories of price $p$ (right) and $Q + \alpha$ in kW (left) for MFG setting (plain lines) compared to MFCagg setting (dotting lines)}
    \label{fig:MFG_MFCagg}
\end{figure}

We also perform numerical computations of the PoA. As expected, the PoA (see Table \ref{tab:PoA_standardprice}) increases with the proportion of customers who have a DSM contract in the population and is strictly superior to 1 when $\pi$ is low enough. When we consider $\pi=0.95$, the impact of the MFG optimization compared to MFC is indeed very little as the proportion of DSM consumers in the population is too low to impact the Price Of Anarchy.\\
 
\begin{table}[H]
\centering
\caption{PoA with standard prices $(p1 = 87.43)$.}
\label{tab:lowp}
\begin{tabular}{lcccc}
\toprule
        & $\pi=0$ & $\pi=0.1$ & $\pi=0.5$ & $\pi=0.95$ \\
\midrule
$V^{MFG}$     &  $34.104$ $(\pm 0.030)$    & $34.354$ $(\pm 0.030)$  &  $35.387$ $(\pm 0.031)$     &  $36.601$ $(\pm 0.032)$ \\
$V^{{MFC}^{agg}}$  &  $33.519$ $(\pm 0.029)$ &  $33.876$ $(\pm 0.029)$ &  $35.226$ $(\pm 0.030)$ & $36.599$ $(\pm 0.032)$        \\
$PoA$     & $1.017465$   & $1.014111$   & $1.004558$    & $1.000072$   \\
\bottomrule
\end{tabular}
\label{tab:PoA_standardprice}
\end{table}

It can also be observed that the PoA is sensitive to the different parameters of the model. In particular, by varying the coefficient $p_1$, we remark that if the spot price becomes much higher, the PoA increases as well as illustrated in Table \ref{tab:PoA_highprice}.   \\
\begin{table}[H]
\centering
\caption{PoA with high prices $(p1 = 1000)$.}
\label{tab:highp}
\begin{tabular}{lcccc}
\toprule
        & $\pi=0$ & $\pi=0.1$ & $\pi=0.5$ & $\pi=0.95$ \\
\midrule
$V^{MFG}$     &  $139.973$ $(\pm 0.127)$    & $144.598$ $(\pm 0.132)$  &  $164.214$ $(\pm 0.157)$     &  $178.435$ $(\pm 0.210)$ \\
$V^{{MFC}^{agg}}$  &  $119.421$ $(\pm 0.097)$ &  $126.039$ $(\pm0.105)$ &  $154.433$ $(\pm 0.145)$ & $179.536$ $(\pm 0.214)$       \\
$PoA$     & 1.172094   & 1.142468   & 1.063333    & 0.993867   \\
\bottomrule
\end{tabular}
\label{tab:PoA_highprice}
\end{table}

\newpage
\printbibliography[
heading=bibintoc,
title={}
] 
\end{document}